\documentclass[11pt]{article}

\usepackage{amsmath,amssymb,amsthm}
\usepackage[unicode,breaklinks=true,colorlinks=true]{hyperref}
\usepackage[dvipsnames]{xcolor}

\usepackage[top=1in, bottom=1in, left=1.25in, right=1.25in, marginparwidth=1in, marginparsep=0.1in]{geometry}

\numberwithin{equation}{section}
\newtheorem{theorem}{Theorem}[section]

\newtheorem{lemma}[theorem]{Lemma}
\newtheorem{definition}[theorem]{Definition}

\theoremstyle{remark}
\newtheorem{remark}[theorem]{Remark}
\newtheorem{example}[theorem]{Example}

\definecolor{darkblue}{rgb}{0,0,0.7}

\newcommand{\bke}[1]{\left( #1 \right)}
\newcommand{\bkt}[1]{\left[ #1 \right]}

\newcommand{\norm}[1]{\left\| #1 \right\|}
\newcommand{\Norm}[1]{\left\Vert #1 \right\Vert}

\newcommand{\al}{\alpha}
\newcommand{\be}{\beta}
\newcommand{\de}{\delta}
\newcommand{\e}{\epsilon}
\newcommand{\ga}{{\gamma}}

\newcommand{\la}{\lambda}

\newcommand{\si}{\sigma}

\newcommand{\ka}{\kappa}
\newcommand{\De}{\Delta}
\renewcommand{\th}{\theta}

\newcommand{\R}{{\mathbb R }}\newcommand{\RR}{{\mathbb R }}
\newcommand{\N}{{\mathbb N}}
\newcommand{\Z}{{\mathbb Z}}

\newcommand{\cN}{{\mathcal N}}

\newcommand{\pd}{{\partial}}
\newcommand{\nb}{{\nabla}}
\newcommand{\lec}{\lesssim}

\newcommand{\I}{\infty}
 
\renewcommand{\div}{\mathop{\mathrm{div}}}

\newcommand{\donothing}[1]{{}}

\newcommand{\EQ}[1]{\begin{equation}\begin{split} #1 \end{split}\end{equation}}
\newcommand{\EQN}[1]{\begin{equation*}\begin{split} #1 \end{split}\end{equation*}}
\newcommand{\Eq}[1]{\begin{equation} #1 \end{equation}}
\newcommand{\Eqn}[1]{\begin{equation*} #1 \end{equation*}}

\DeclareMathOperator*{\esssup}{ess\,sup}

\makeatletter
\newcommand{\xRightarrow}[2][]{\ext@arrow 0359\Rightarrowfill@{#1}{#2}}
\makeatother

\newcommand{\loc}{\mathrm{loc}} 
\newcommand{\uloc}{\mathrm{uloc}}
\newcommand{\far}{\mathrm{far}} 
\newcommand{\near}{\mathrm{near}} 

\newcommand{\LE}{\mathbf{LE}}

\let\OLDthebibliography\thebibliography
\renewcommand\thebibliography[1]{
  \OLDthebibliography{#1}
  \setlength{\parskip}{1pt}
  \setlength{\itemsep}{1pt plus 0.3ex}
}

\begin{document}
\title{Local energy solutions to the Navier-Stokes equations in Wiener amalgam spaces}
\author{Zachary Bradshaw and Tai-Peng Tsai}
\date{\today}
\maketitle 

\begin{abstract}
We establish existence of solutions in a scale of classes weaker than the finite energy Leray class and stronger than the infinite energy Lemari\'e-Rieusset class. The new classes are based on the $L^2$ Wiener amalgam spaces.  Solutions in the classes closer to the Leray class are shown to satisfy some properties known in the Leray class but not  the Lemari\'e-Rieusset class, namely eventual regularity and long time estimates on the growth of the local energy.  In this sense, these solutions bridge the gap between Leray's original solutions and Lemari\'e-Rieusset's solutions and help identify scalings at which certain properties may break down. 
\end{abstract}

\section{Introduction}\label{sec.intro}

The Navier-Stokes equations describe the evolution of a viscous incompressible fluid's velocity field $u$ and associated scalar pressure $p$.  In particular, $u$ and $p$ are required to satisfy
\EQ{\label{eq.NSE}
&\partial_tu-\Delta u +u\cdot\nabla u+\nabla p = 0,
\\& \nabla \cdot u=0,
}
in the sense of distributions.  For our purpose, \eqref{eq.NSE} is applied on $\R^3\times (0,\I)$ and $u$ evolves from a prescribed, divergence free initial data $u_0:\R^3\to \R^3$.

In the classical paper \cite{leray}, J.~Leray constructed  global-in-time weak solutions to \eqref{eq.NSE} on $\R^4_+=\R^3\times (0,\infty)$ for any divergence free vector field $u_0\in L^2(\R^3)$.  Leray's solution $u$ satisfies the following properties:
\begin{enumerate}
\item $u\in L^\I(0,\I;L^2(\R^3))\cap L^2(0,\I;\dot H^1(\R^3))$,
\item $u$ satisfies the weak form of \eqref{eq.NSE},
\[
\iint -u \pd_t \zeta + \nb u:\nb \zeta + (u \cdot \nb )u \cdot \zeta = 0,\quad \forall \zeta \in C^\infty_c(\R^4_+;\R^3), \quad \div \zeta=0,
\]
\item $u(t)\to u_0$ in $L^2(\R^3)$ as $t\to 0^+$,
\item $u$ satisfies the \emph{global energy inequality}: For all $t>0$,
\[
\int_{\R^3} |u(x,t)|^2\,dx +2 \int_0^t \int_{\R^3} |\nabla u(x,t)|^2\,dx\,ds \leq \int_{\R^3} |u_0(x)|^2\,dx.
\]
\end{enumerate}
 The above existence result was extended to domains by Hopf in \cite{Hopf}.
We refer to the solutions constructed by Leray as \emph{Leray's original solutions} and refer to 
any solution satisfying the above properties as a \emph{Leray-Hopf weak solution}.   Note that, based on their construction, Leray's original solutions satisfy additional properties.  For example, they are suitable in the sense of \cite{CKN}; see \eqref{CKN-LEI}, this is proven in \cite[Proposition 30.1]{LR}. Leray-Hopf weak solutions, on the other hand, are not known to be suitable generally. 

Although many important questions about these weak solutions remain open, e.g.,~uniqueness and global-in-time regularity, some positive results are available.  In particular, it is known that the singular sets of Leray-Hopf weak solutions which are suitable are precompact in space-time.  This follows  from Leray \cite[(6.4)]{leray}, and the partial regularity results of Scheffer \cite{VS76b} and Caffarelli, Kohn, and Nirenberg \cite{CKN} (see also \cite{LR}, \cite{BCI}, and \cite[Chap.~6]{Tsai-book}).

In his book \cite{LR}, Lemari\'e-Rieusset introduced a local analogue of suitable Leray-Hopf weak solutions called \emph{local energy solutions}.  These solutions evolve from uniformly locally square integrable data  $u_0\in L^2_{\uloc}$. Here, for $1\le q \le \infty$, $L^q_{\uloc}$ is the space of functions on $\R^3$ with finite norm
\[
\norm{u_0}_{L^q_{\uloc}} :=\sup_{x \in\R^3} \norm{u_0}_{L^q(B(x,1))}<\infty.
\]
We also denote
\[
E^q = \overline{C_c^\I(\R^3)}^{L^q_{\uloc}},
\]
the closure of $C_c^\I(\R^3)$ in $L^q_{\uloc}$-norm.
Having a notion of weak solution in a broader class than Leray's is useful when analyzing initial data in critical spaces such as the Lebesgue space $L^3$, the Lorentz space $L^{3,\I}=L^3_w$, or the Morrey space $M^{2,1}$, all of which embed in $L^2_{\uloc}$ but not in $L^2$ (see \cite{JiaSverak} for an example where this was crucial).  By \emph{critical spaces} we mean spaces for which the norm of $u$ is scaling invariant.  It is in such spaces that many arguments break down.  For example, $L^\I(0,T;L^3)$ is a regularity class for Leray-Hopf solutions \cite{ESS}, but this is unknown for $L^\I(0,T;L^3_w)$. 

The following definition is motivated by those found in \cite{LR,KiSe,JiaSverak-minimal,JiaSverak}.

\begin{definition}[Local energy solutions]\label{def:localLeray} Let $0<T\leq \I$. A vector field $u\in L^2_{\loc}(\R^3\times [0,T))$ is a local energy solution to \eqref{eq.NSE} with divergence free initial data $u_0\in L^2_{\uloc}(\R^3)$, denoted as $u \in \cN(u_0)$, if:
\begin{enumerate}
\item for some $p\in L^{3/2}_{\loc}(\R^3\times [0,T))$, the pair $(u,p)$ is a distributional solution to \eqref{eq.NSE},
\item for any $R>0$, $u$ satisfies
\begin{equation}\notag
\esssup_{0\leq t<R^2\wedge T}\,\sup_{x_0\in \R^3}\, \int_{B_R(x_0 )}\frac 1 2 |u(x,t)|^2\,dx + \sup_{x_0\in \R^3}\int_0^{R^2\wedge T}\int_{B_R(x_0)} |\nabla u(x,t)|^2\,dx \,dt<\infty,\end{equation}
\item for any $R>0$, $x_0\in \R^3$, and $0<T'< T $, there exists a function of time $c_{x_0,R}(t)\in L^{3/2}(0,T')$ so that, for every $0<t<T'$  and $x \in B_{2R}(x_0)$  
\EQ{ \label{pressure.dec}
&p(x,t)=-\Delta^{-1}\div \div [(u\otimes u )\chi_{4R} (x-x_0)]
\\&\quad - \int_{\R^3} (K(x-y) - K(x_0 -y)) (u\otimes u)(y,t)(1-\chi_{4R}(y-x_0))\,dy 
+ c_{x_0,R}(t),
}
in $L^{3/2}(B_{2R}(x_0)\times (0,T'))$
where  $K(x)$ is the kernel of $\Delta^{-1}\div \div$,
 $K_{ij}(x) = \pd_i \pd_j \frac {-1}{4\pi|x|}$, and $\chi_{4R} (x)$ is the characteristic function for $B_{4R}$. 
\item for all compact subsets $K$ of $\R^3$ we have $u(t)\to u_0$ in $L^2(K)$ as $t\to 0^+$,
\item $u$ is suitable in the sense of Caffarelli-Kohn-Nirenberg, i.e., for all cylinders $Q$ compactly supported in  $ \R^3\times(0,T )$ and all non-negative $\phi\in C_c^\infty (Q)$, we have  the \emph{local energy inequality}
\EQ{\label{CKN-LEI}
&
2\iint |\nabla u|^2\phi\,dx\,dt 
\\&\leq 
\iint |u|^2(\partial_t \phi + \Delta\phi )\,dx\,dt +\iint (|u|^2+2p)(u\cdot \nabla\phi)\,dx\,dt,
}
\item the function
\[
t\mapsto \int_{\R^3} u(x,t)\cdot {w(x)}\,dx
\]
is continuous in $t\in [0,T)$, for any compactly supported $w\in L^2(\R^3)$.
\end{enumerate}
For a given divergence free $u_0\in L^2_{\uloc}$, let $\mathcal{N}(u_0)$ denote the set of all local energy solutions with initial data $u_0$.
\end{definition}

The constant $c_{x_0,R}(t)$ can depend on $T'$ in principle. This does not matter in practice and we omit this dependence.

Our definition of local energy solutions
 is slightly different than the definition from \cite{LR,KiSe,JiaSverak-minimal,JiaSverak}.  The definition used in \cite{KiSe,JiaSverak-minimal,JiaSverak} requires the data be in $E^2$,  which implies some very mild decay at spatial infinity.  The pressure representation \eqref{pressure.dec} is replaced in \cite{JiaSverak-minimal,JiaSverak} by a very mild decay assumption on $u$, namely
\[
\lim_{|x_0|\to \I} \int_0^{R^2}\int_{B_R(x_0)} |u(x,t)|^2\,dx\,dt=0, \quad \forall R>0 .
\]
This condition implies a pressure representation like \eqref{pressure.dec} is valid (this is mentioned in \cite{JiaSverak-minimal} and explicitly proven in \cite{MaMiPr,KMT}).  If the data is only in $L^2_{\uloc}$, the above decay condition is unavailable and, therefore, we must build the pressure formula into the definition. This rules out `parasitic' solutions.
In this paper we work exclusively in subspaces of $E^2$, so this distinction is not relevant.

In \cite{LR} (also see \cite{LR2}), Lemari\'e-Rieusset constructed 
local in time  local energy solutions if $u_0$ belongs to $L^2_\uloc$, and
global in time local energy solutions if $u_0$ belongs to $E^2$.
Kikuchi and Seregin \cite{KiSe} constructed global solutions for data in $E^2$ with more details and prove they satisfy the pressure formula in Definition \ref{def:localLeray} but with $R=1$.   
Recently, Maekawa, Miura, and Prange constructed local energy solutions on the half-space \cite{MaMiPr}.  This is a non-trivial extension of the whole-space case and required a novel treatment of the pressure.

When there is no spatial decay, the global existence problem is generally open. Some partial results have been established. Kwon and Tsai \cite{KwTs} constructed global in time local energy solutions for non-decaying $u_0$ in $L^3_\uloc+ E^2$ with slowly decaying oscillation. Bradshaw and Tsai \cite{BT8}, Fern\'andez-Dalga and Lemari\'e-Rieusset \cite{FDLR2} and Bradshaw, Kukavica and Tsai \cite{BKT} all constructed global solutions with non-decaying or even growing data in weighted spaces.

Naturally, less is known about local energy solutions than Leray-Hopf weak solutions.  For example, Leray-Hopf weak solutions that satisfy the local energy inequality have singular sets that are precompact. Leray proved this in \cite[paragraph 34]{leray}, giving an upper bound of the set of singular times in \cite[(6.4)]{leray}.  Analogous results are currently unavailable for local energy solutions. Indeed, it is speculated in \cite{BT1} that eventual regularity does not hold  for a discretely self-similar solution with $u_0 \in L^{3,\infty}(\R^3)$ if it has a local singularity. Note that eventual regularity of local energy solutions has recently been studied in \cite{BT8} where conditions are given on the initial data ensuring eventual regularity holds.
Examining the gap between Leray's original solutions and local energy solutions is the main motivation for this paper.  To better understand the properties of weak solutions beyond the suitable Leray-Hopf class,
we introduce a scale of initial data spaces which connect $L^2$  and $L^2_\uloc$. We establish global existence of solutions for these scales and show that, in the spaces close to $L^2$, the new solutions share some properties with Leray-Hopf solutions.

We now introduce a scale of spaces $E^2_q$ where $1\leq q\leq \I$ connecting $L^2$ and $E^2$ defined as follows:  For $q<\I$,
 $u_0\in E^2_q$ if and only if
\[
\|u_0\|_{E^2_q} :=\bigg\|   \bigg( \int_{B_1(k)}|u_0(x)|^2\,dx  \bigg)^{1/2} \bigg\|_{l^q(k\in \Z^3)}<\I.
\]
We identify $E^2_\I$ with $E^2$. Clearly, $E^2_2=L^2$, $C^\infty_0(\R^3)$ is dense in $E^2_q$ if $q\le \I$,
 and $E^2_s \subset E^2_q$ with
$\|u_0\|_{E^2_q} \le \|u_0\|_{E^2_s} $ if $1\le s\le q\le \I$.
The spaces $E^2_q$ form a subclass of the \emph{Wiener amalgam spaces}, which treat local and global behaviors separately, see \cite{FoSt, Holland,KNTYY,CKS,GWYZ} and their references. Indeed, it is interesting to note their  connection to the $L^2$-based Besov spaces:  For $q>2$,
$
B^s_{2,q} \subset E^2_q $ if $s>0$ and $
E^2_q \subset B^s_{2,q}$ if $s< n (\frac 1q-\frac 12)
$, \cite[Corollary 1.2]{CKS}.

We introduce more notation to represent the class of functions of interest to us. Let $I=(t_0,t_1)$ be an open interval in $\R_+$. Denote by $\LE_q(t_0,t_1)$ the class of functions with finite norm
\Eq{\label{LEq.def}
\norm{u}_{\LE_q(I)} =  \bigg\|\bke{ \esssup_{t \in I} \int_{B_1(k) }|u(x,t)|^2  \,dx+ \int_I \int_{B_1(k) } |\nabla u|^2\,dx\,dt }^{1/2}\bigg\| _{l^{q}(k\in \Z^3)}.
}
We also denote the first part of the norm as 
\Eq{\label{LEqflat.def}
\norm{u}_{\LE^\flat_q(I)} =  \bigg\|\bke{ \esssup_{t \in I} \int_{B_1(k) }|u(x,t)|^2  \,dx }^{1/2}\bigg\| _{l^{q}(k\in \Z^3)}.
}
The notation $\LE_q$ means (the square root of) the \emph{local energy} at sites $k$ is in $l^q(k)$.
If $I$ is omitted we assume $I=(0,\I)$ unless the context suggests otherwise.
If we denote
\[
a_k(u) = \esssup_{t \in I} \int_{B_1(k) }|u(x,t)|^2  \,dx, \quad
b_k(u) = \int_I \int_{B_1(k) } |\nabla u|^2\,dx\,dt,
\]
then
\[
\norm{u}_{\LE_q(I)} = \norm{a_k(u)+b_k(u)} _{l^{q/2}(k\in \Z^3)} ^{1/2}.
\]
Note that $\norm{a_k(u)} _{l^{q/2}(k\in \Z^3)}<\infty$ is stronger than $u \in L^\infty(I; E^2_q)$ since
\[
\norm{u(t)}_{E^2_q} \le \norm{a_k(u)} _{l^{q/2}(k\in \Z^3)} ^{1/2}, \quad \forall t\in I.
\]
The reverse inequality is wrong. 

\begin{example}\label{example1.2}
A function $u \in L^\infty(I; E^2_q)$ may not have $a_k(u) \in l^{q/2}$. As an example, fix a smooth function $\phi$ supported in $B_1$.
For $j \in \N_0$ take $I_j = (2^{-j-1},2^{-j}]$ and $x_j = (2^j,0,0)$. Let
$u(x,t) = \phi(x-x_j)$ if $x \in B_1(x_j)$ and $t \in I_j$ for some $j\in\N_0$, $u(x,t)=0$ otherwise. Then 
$\norm{u(t)}_{E^2_q}$ is constant in $t$ for any $q$, and $u \in L^\infty L^2 \cap L^2 H^1(\R^3\times I)$, $I=(0,1)$. But $a_k(u)=C$ if $k = x_j$, thus $\{ a_k(u)\}_{k \in \Z^3} \not \in  l^{r}$ for any $r<\infty$.
\end{example}
 
For a solution $u$ in $\R^4_+$, we say that $(x,t)$ is a \emph{singular point} of $u$ if $ u \notin L^\I(B(x,r)\times(t-r^2,t))$ for any $r>0$. 
The set of all singular points is the \emph{singular set} of $u$.
We say that $t$ is a \emph{singular time} if there is a singular point $(x,t)$ for some $x$.  We say a solution $u$ has \emph{eventual regularity} if there is $t_1 < \infty$ such that $u$ is regular at $(x,t)$ whenever $t_1\le t$.  We say $u$ has \emph{initial regularity} if there exists $t_2$ such that $u$ is regular at $(x,t)$ whenever $0<t<t_2$. 

Our first result establishes eventual regularity for local energy solutions with data in the $L^2$-based Wiener-amalgam spaces close to $L^2$.
 
\begin{theorem}[Eventual regularity in $E^2_q$]\label{thrm.ER.E2q}
Assume $u_0\in E^2_q$ where $2\leq q\leq 3$, is divergence free and $u\in \mathcal N(u_0)$.  Then $u$ has eventually regularity and 
\[\|u(\cdot,t)\|_{L^\I}\lesssim t^{1/2},\]
for sufficiently large $t$.
\end{theorem}

When $q>3$, the premises of Theorem \ref{thrm.ER.E2q} break down and we cannot prove eventual regularity.  This makes sense because $L^{3,\I}\subset E^2_{3+}$ (see the appendix) and we do not expect initial or eventual regularity in $L^{3,\I}$---see discussion in \cite{BT1}. However, an explicit long-time upper bound on the growth of the scaled $E^2_q$ norms can still be obtained when $q< 6$. This is interesting as it gives a new $L^2$-based estimate that is sensitive to the decay properties of the initial data. Existing estimates only measure the growth of the $L^2_\uloc$ norm and do not keep track of summability. Furthermore, bounds in terms of the initial data generally only extend up to a finite time.

\begin{theorem}[Explicit growth rate in $E^2_q$]\label{thrm.boundE2q}
Assume $u_0\in E^2_q$ where $2\leq q< \I$,  
is divergence free and 
$u\in \mathcal N(u_0)$ satisfies, for some  $T_2>0$, 
 \[
\|u\|_{\LE_q(0,T_1)}
 <\I, \quad \forall T_1 \in (0,T_2).
\]
Then, for any $R\ge1$,
 with $T= \min \big( \la_1(1+ \norm{u_0}_{E^2_q})^{-4} R^{\min (2, 12/q -2)}, \, T_2\big)$, we have
\begin{align*}
&	\bigg\|	\sup_{0\leq t\leq T} \int_{B_R(Rk)} |u(x,t)|^2\,dx + \int_0^{T}\int_{B_R(Rk)} |\nb u(x,t)|^2\,dx\,dt \bigg\|_{l^{\frac q 2}(k\in \Z^3)}
\leq C\|u_0\|_{E^2_q}^2 R^{3-\frac 6 q},
\end{align*}
for positive constants $\la_1$ and $C$ independent of $u_0$ and $R$.
In particular, if $T_2=\I$ and $q<6$ then $T\to \I$ as $R\to \I$.
\end{theorem}

Predictably, this estimate is uniform in time when $u_0\in L^2$, in which case our solution has finite energy.  
Because the $E^2_q$ spaces form a ladder between $L^2$ and $E^2$, Theorem \ref{thrm.boundE2q} gives a precise statement of how the global-in-time energy bound for Leray-Hopf weak solutions breaks down in adjacent infinite energy classes.  In particular, time-global estimates are available when $q<6$.

This estimate may prove useful in other contexts.  For example, since the endpoint Lorentz space $L^{3,\I}\subset E^2_{3+}$, Theorem \ref{thrm.boundE2q} gives a new a priori bound for local energy solutions in $\LE_q$ with data in $L^{3,\I}$.

We emphasize that we have identified two interesting parameters that determine the properties of local energy solutions:
\begin{itemize}
\item If $q\leq 3$, then the solution eventually regularizes, a property shared with Leray's weak solutions.  In essence, this initial data is locally $L^2$ but has \emph{critical or supercritical} decay at spatial infinity.
\item If $q<6$, then information about the growth of the local energy can be extended to arbitrarily large times, a property resembling the global in time bound on the $L^2$ norm of Leray's weak solutions. 
\end{itemize}

These findings reflect recent results in \cite{BT8}. In that paper we considered the quantity
\[
Q_s(R)=\lim_{R\to \I}\sup_{x_0\in \R^3} \frac 1 {R^s} \int_{B_R(x_0)}|  u_0(x)|^2\,dx,
\]
and found that, if  $s\leq 1$ and $Q_s(R)\to 0$, then a solution has eventual regularity.  Alternatively, if $s\leq 2$ and $Q_s(R)\to 0$, then certain a priori bounds can be extended to arbitrarily large times, a fact which allows us to construct global in time local energy solutions. In both cases,  the endpoint cases match the scaling of the results summarized above.

Establishing the estimates in  Theorem \ref{thrm.boundE2q}  require $u\in \LE_q$, which is unclear for existing local energy solutions. This naturally raises the question of existence in the $\LE_q$ class.

\begin{theorem}[Existence in $E^2_q$]\label{thrm.existence} 
Assume $u_0\in E^2_q$ where $2\leq q<\I$ and is divergence free.  Then,  
 there exists a time-global local energy solution $u$  
and associated pressure $p$ having initial data $u_0$ so that, for any $0<T<\infty$, 
\[
\|u\|_{\LE_q(0,T)}<\I.
\]
In particular, $u\in L^\I(0,T;E^2_q)$.
\end{theorem}  

Since $E^2_q$ embeds in $E^2$, a space for which global existence is known, the important part of Theorem \ref{thrm.existence} is the   bounds in $\LE_q$. 
 Indeed, it is not clear that a generic local energy solution with data in $E^2_q$ satisfies these bounds.
The bulk of this paper is dedicated to proving Theorem \ref{thrm.existence}. This is because, although $E^2_q$ shares structural elements with $L^2$ and $E^2$, these break down for $\LE_q$ in comparison to $L^\I L^2\cap L^2H^1$ and $L^\I L^2_\uloc\cap \widetilde {L^2 \dot H^1_\uloc}$, where $\|u\|_{\widetilde {L^2 \dot H^1_\uloc}}^2= \sup_{x_0\in \R^3} \int_0^T \int_{B_1(x_0)} |\nb u|^2\,dx\,ds$ (these are respectively the solution classes for Leray-Hopf solutions and local energy solutions). 

When $q<6$,
the bounds in Theorem \ref{thrm.existence} are {\bf a priori}, i.e., depending only on $u_0$, by Theorem \ref{thrm.boundE2q}. They can be used for alternative construction of global solutions as limits of $u^k$ defined in $(0,T_k)$, $T_k \to \infty$, 
in the same way as in \cite{BT8}. We do not know if we have a priori bounds up to time infinity if $q\ge 6$, hence we do not adapt it in this paper.

\medskip 
This paper is organized as follows. In Section \ref{sec.ER} we prove Theorem \ref{thrm.ER.E2q}.  In Section \ref{sec.bound} we give new apriori bounds in the $\LE_q$ class and prove Theorem \ref{thrm.boundE2q}. In Section \ref{sec.local} we construct local solutions in $\LE_q$. These solutions are extended to global solutions in Section \ref{sec.global}.  We include two appendixes containing elementary or known results which the reader may nonetheless find convenient.
The first examines the local existence of strong solutions when the data is in $E^4$ and the second presents helpful remarks on the relationships between the $L^2$-based Wiener amalgam spaces and the endpoint critical Lorentz space $L^{3,\I}$ which is important for the Navier-Stokes problem.

\section{Eventual regularity}\label{sec.ER}

The following result is contained in \cite{BT8} and will be used to prove Theorem \ref{thrm.ER.E2q}.
\begin{theorem}[\cite{BT8}]\label{thrm.BT8}
There is a small positive constant $\e_1$ such that the following holds. Assume   $u_0\in L^2_{\uloc}(\R^3)$, is divergence free and $u\in \mathcal N(u_0)$. Let 
\[
N_R^0 := \sup_{x_0\in \R^3} \frac 1 R \int_{B_R(x_0)}|  u_0|^2\,dx.
\]
If there exists $R_0>0$ so that
\begin{align}
\label{cond.infinity} \sup_{R\geq R_0} N^0_R< \e_1,
\end{align} 
then $u$ has eventual regularity.  Moreover, if $R_0^2\lesssim t$, then
\[
t^{1/2}\|u(\cdot,t)\|_{L^\I}\lesssim  ( \sup_{R\geq R_0} N^0_R)^{1/2} <\I.
\] 
\end{theorem}

The proofs of Theorems \ref{thrm.ER.E2q} and \ref{thrm.boundE2q} use the following lemma. It involves the quantity
\[
N^0_{q,R}(u_0)  =  \frac 1 R   \bigg(  \sum_{k\in \Z^3} \bigg( \int_{B_R(kR)}|u_0|^2\,dx \bigg)^{q/2} 			\bigg)^{2/q},
\quad
{N^0_{\I,R}(u_0)=N^0_{R}(u_0)},
\]
which will appear again in Lemma \ref{lem.A0qbound}.  Note that $N^0_{q,R} \le N^0_{s,R}$ if $1\le s\le q\le \I$.

\begin{lemma}\label{lemma.E2q.limits}Assume $u_0\in E^2_q$ where $2 <  q< \I$.  Then,
\[
\lim_{R\to \I}  R^{6/q-2} N^0_{q,R}(u_0)=0.
\]
Consequently, if $u_0\in E^2_q$, then
\[
\lim_{R\to \I}   N^0_{R}(u_0)=0 \mbox{ if }2\leq q\leq 3\mbox{ and }\lim_{R\to \I}   R^{-1} N^0_{q,R}(u_0)=0\mbox{ if }2\leq q\leq 6. 
\]
\end{lemma} 
The first part of the lemma excludes $q=2$ as $R N^0_{2,R} (u_0)\sim \int_{\R^3} | u_0|^2 \sim \norm{u}_{E^2_2}^2$. 
\begin{proof}

Let $\e>0$ be given.
If $u_0 \in E^2_q$, $2 \le q < \I$,  then we have $\norm{u_0}_{E^2_q} = \norm{a}_{ l^q(\Z^3)}$ where
\[
a = (a_k)_{k \in \Z^3} \in  l^q(\Z^3), \quad a_k = \norm{u_0}_{L^2(B_1(k))}.
\]
For $R\ge1$, 
\EQ{\label{NqRbound}
N_{q,R}^0(u_0) & 
\le \frac C R   \bigg(  \sum_{k\in \Z^3} \bigg(  \sum_{|i-kR|<R} a_i^2 \bigg)^{q/2} 	\bigg)^{2/q}.
}
For any $\de>0$, we can choose $M>1$ such that $ \norm{a^{>M} }_{ l^q}\le \de$, where
\[
a^{>M}_k = \begin{cases}
 0 & \text{if }|k|\le M
\\a_k& \text{if }|k|> M 
\end{cases}.
\]
Let $a^{\leq M} = a - a^{>M}$. 
By H\"older's inequality we have,  for $R>M$,
\EQ{
\bigg(\sum_{|i-kR|<R} a_i^2 \bigg)^{q/2}&\leq 
C\bigg(\sum_{|i-kR|<R}  (a_i^{> M})^2\bigg)^{q/2}
+C\bigg(\sum_{|i-kR|<R}  (a_i^{\le M})^2\bigg)^{q/2}
\\&\leq C R^{3(q-2)/2}  \sum_{|i-kR|<R} (a_i^{>M})^q 	+    C M^{3(q-2)/2}  \sum_{|i-kR|<R} (a_i^{\leq M})^q.
}
Thus
\EQN{
N_{q,R}^0(u_0)^{q/2} &  \leq \frac C {R^{q/2}}     \sum_{k\in \Z^3}  
\bigg(   
 R^{3(q-2)/2}  \sum_{|i-kR|<R} (a_i^{>M})^q 	+      M^{3(q-2)/2}  \sum_{|i-kR|<R} (a_i^{\leq M})^q		
\bigg)  
\\&\leq CR^{3(q-2)/2-q/2}\|a^{>M}\|_{l^q}^q + CR^{-q/2}M^{3(q-2)/2}\|a^{\leq M}\|_{l^q}^q.
}
Thus
\EQ{
\big[R^{6/q-2}N_{q,R}^0(u_0)\big]^{q/2}&  \le C\|a^{>M}\|_{l^q}^q+CR^{ {3-3q/2}} M^{3(q-2)/2}\|a^{\leq M}\|_{l^q}^q.
}
If we first choose $\de$ sufficiently small we can ensure $C\|a^{>M}\|_{l^q}^q<\e/2$.  Then, taking $R$ sufficiently large ensures that $CR^{ {3- 3q/2}} M(\de)^{3(q-2)/2}\|a^{\leq M}\|_{l^q}^q<\e/2$, also,  if $q>2$.

To prove the last statements, first note that $N^0_R(u_0)\leq N^0_{q,R}(u_0)$. Also, $u_0\in E^2_q$ for $q\leq 3$ implies $u_0\in E^2_3$.  Hence,  
\[
	\lim_{R\to \I}N^0_R(u_0)\leq \lim_{R\to \I} N^0_{3,R}(u_0) = 0.
\] 
For the last statement, note that when $q\leq 6$ and $R\geq 1$, we have $R^{-1}\leq R^{6/q-2}$.   Hence
\[
	\lim_{R\to \I} R^{-1}N_{q,R}^0(u_0)=0.\qedhere
\]
\end{proof}

\begin{proof}[Proof of Theorem \ref{thrm.ER.E2q}]
By Lemma \ref{lemma.E2q.limits}, if $q\leq 3$ we have \[\lim_{R\to \I}   N^0_{R}(u_0)=0 .\]
Then, use Theorem \ref{thrm.BT8} to obtain the desired conclusion.
\end{proof}

\begin{remark}\label{remark.e2qdata}  In Theorem \ref{thrm.ER.E2q}, the exponent $q=3$ is sharp in the sense that the premises of the theorem on eventual regularity from \cite{BT8} are not implied when $u_0 \in E^2_q\setminus E^2_3$ for $q>3$.  Consider for example $|x|^{-1}$.  Letting $a_k=\int_{B_1(k)}|x|^{-2}\,dx$, we have $a_k\sim (1+|k|^2)^{-1}$.  Then
\[
\||x|^{-1}\|_{E^2_3}^3 \sim \sum_{k\in \Z^3} a_k^{3/2} \sim  \sum_{k\in \Z^3}  (1+|k|^2)^{-3/2},
\]
which diverges.  Looking ahead to Lemma \ref{lem.alphasum}, we know $L^{3,\I}\subset E^2_{3+}$ and, therefore, $|x|^{-1}\in E^2_{3+}$.  Furthermore, $N^0_R(|x|^{-1})$ does not vanish as $R\to \I$.  

\end{remark}

\section{A priori bounds for some local energy solutions}\label{sec.bound}

In this section we  prove new a priori bounds for data $u_0 \in E^2_q$ and use it to prove Theorem \ref{thrm.boundE2q}. 
To motivate the a priori bound, we recall  a well known bound for local energy solutions (see \cite[Lemma 2.2]{JiaSverak-minimal}, 
for all $u\in \mathcal N (u_0)$ and $r>0$ we have
\begin{equation}\label{ineq.apriorilocal}
\esssup_{0\leq t \leq \sigma r^2}\sup_{x_0\in \RR^3} \int_{B_r(x_0)}\frac {|u|^2} 2 \,dx\,dt + \sup_{x_0\in \RR^3}\int_0^{\sigma r^2}\int_{B_r(x_0)} |\nabla u|^2\,dx\,dt <CA_0(r) ,
\end{equation}

\begin{equation}\label{ineq.apriorilocal2}
\sup_{x_0\in \RR^3} \int_0^{\sigma r^2}\!\!\int_{B_r(x_0) }\big( | u|^3  +|p-c_{x_0,r}(t)|^{3/2}  \big)\,dx\,dt
 <C r^{\frac 12} A_{0}(r)^{\frac 32},
\end{equation}
where
\[
A_0(r)=rN^0_r= \sup_{x_0\in \R^3} \int_{B_r(x_0)} |u_0|^2 \,dx,
\] 
and
\begin{equation}\label{def.sigma}
\si=\sigma(r) =c_0\, \min\big\{(N^0_r)^{-2} , 1  \big\},
\end{equation}
for a small universal constant $c_0>0$.  Care is required here because,  as mentioned in Section \ref{sec.intro}, the solutions in \cite{JiaSverak-minimal} are defined differently than they are here--we only require $u_0 \in L^2_{\uloc}$ and do not require $u_0 \in E^2$, and therefore assume \eqref{pressure.dec} explicitly.  Inspecting  \cite[Proof of Lemma 2.2]{JiaSverak-minimal}, however, reveals that the same conclusion is valid for our local energy solutions.  In particular, the only reason to assume $u_0\in E^2$ is that it implies \eqref{pressure.dec}. See \cite[Lemma 3.5]{KMT} for revised \eqref{ineq.apriorilocal2} with higher exponents.

Our bound is a refinement of \eqref{ineq.apriorilocal} when the initial data has more decay at spatial infinity.

\begin{lemma}\label{lem.A0qbound}
Assume $u_0\in E^2_q$ for some $2\leq q<\I$ 
 is divergence free and that  $u\in \mathcal N(u_0)$ satisfies, for some  $T_2>0$, 
 \begin{equation}\label{th2.2-0}
 \bigg\|\esssup_{0\leq t \leq T_1} \int_{B_1(x_0) }|u|^2  \,dx+ \int_0^{T_1}\int_{B_1(x_0) } |\nabla u|^2\,dx\,dt \bigg\| _{l^{q/2}(x_0\in \Z^3)}  <\I, \quad \forall T_1 \in (0,T_2).
\end{equation}
Then there are positive constants $C_1$ and $\la_0<1$, both independent of $q$ and $R$ such that, for all $R>0$ with $\la_R R^2\le T_2$,
\begin{equation}\label{th2.2-1}
\bigg\|\esssup_{0\leq t \leq \la_R R^2} \int_{B_R(x_0R) }\frac {|u|^2}2  \,dx+ \int_0^{\la_R R^2}\int_{B_R(x_0R) } |\nabla u|^2\,dx\,dt \bigg\| _{l^{q/2}(x_0\in \Z^3)}
\leq  C_1 A_{0,q}(R),
\end{equation}
where 
\[
A_{0,q}(R) = R N^0_{q,R} =  \bigg\|  \int_{B_R(x_0R)  }  |u_0(x)|^2 \,dx   \bigg\|_{l^{q/2}(x_0\in \Z^3)} ,
\quad   \la_R =\min (\la_0, \frac{\la_0 R^2}{A_{0,q}(R)^{2}}). 
\]
Furthermore, for all $R>0$,  
\begin{equation}\label{th2.2-2}
\bigg\|\int_0^{\la_R R^2}\!\!\int_{B_R(x_0R) } 
|u|^{\frac {10}3} +|p-c_{Rx_0,R}(t)|^{\frac53}\,dx\,dt \bigg\|_{l^{\frac {3q}{10}}(x_0 \in \Z^3)}
\le  C A_{0,q}(R)^{\frac 53}.
\end{equation}
 \end{lemma}

\begin{remark}
This lemma is true for any $R>0$. Indeed, it can be proved by rescaling the result for $R=1$. 
 By \eqref{th2.2-2} and H\"older inequality, we also have
\begin{equation}\label{th2.2-3}
\bigg\|\int_0^{\la_R R^2}\!\!\int_{B_R(x_0R) }\big( | u|^3  +|p-c_{Rx_0,R}(t)|^{3/2}  \big)\,dx\,dt \bigg\|_{l^{\frac q3}(x_0\in \Z^3)} 
<C \la_R^{\frac 1{10}} R^{\frac 12} A_{0,q}(R)^{\frac 32}.
\end{equation}
It can be also shown by direct estimates similar to the proof of \eqref{th2.2-2} without using H\"older inequality, with the factor $\la_R^{\frac 1{10}}$ replaced by a smaller $\la_R^{1/4}$.
\end{remark}

 \begin{proof} 

Let $\phi_0\in C_c^\I(\R^3)$ be radial, non-increasing, identically $1$ on $B_1(0)$, supported on $B_2(0)$, and satisfy $|\nabla \phi_0(x)|\lesssim  1$ and $|\nabla \phi_0^{1/2}(x)|\lesssim 1$.  Let $R>0$ be as in the statement of the lemma. Let $\phi(x)=\phi_0(x/R)$.  Let $0 < \la \le 1$.

For $\kappa\in R\Z^3$, let \begin{align*}
e_{R,\la}(\ka):=& \esssup_{0\leq t\leq \la R^2}\int |u(t)|^2 \phi(x-\ka)   \,dx +\int_0^{\la R^2}\!\! \int |\nb u|^2 \phi(x-\ka)  \,dx\,dt.
\end{align*}
We will begin by establishing bounds on $e_{R,\la}(\ka)$ and then use these to bound the quantity
\begin{align*}
E_{R,q,\la}&:= \bigg\|   \esssup_{0\leq t\leq \la R^2}\int |u(t)|^2 \phi(x-Rk)   \,dx 
+\int_0^{\la R^2}\!\int |\nabla u|^2 \phi(x-Rk)  \,dx\,ds \bigg\|_{l^{q/2}(k\in \Z^3)}^{q/2},
\end{align*}
in terms of $A_{0,q}(R)$
for  sufficiently small $\la$.  By assumption, $E_{R,q,\la}<\infty$.
Our starting point for bounding $e_{R,\la}(\ka)$ is the local energy inequality 
\EQ{\label{ineq.lei.zero}
&\int |u(t)|^2 \phi(x-\ka)   \,dx +2\int_0^t \int |\nabla u|^2 \phi(x-\ka)  \,dx\,ds
\\&\leq \int |u_0|^2\phi(x-\ka)  \,dx+ \int_0^t\int |u|^2 \Delta\phi(x-\ka)  \,dx\,ds 
\\&\quad+\int_0^t\int |u|^2 (u \cdot \nabla\phi(x-\ka) )\,dx\,ds
+\int_0^t\int2 p (u\cdot \nabla\phi(x-\ka) )\,dx\,ds,
} which holds because $u$ is a local energy solution.
We proceed term by term starting with the  second.  Using the properties of $\phi$ we have
\begin{align*}
\int_0^{\la  R^2}\int |u|^2 |\Delta  \phi(x-\ka)|  \,dx\,dt 
&\leq \frac C {R^2} \int_0^{\la  R^2} \int_{B_{2R}( \ka)} |u|^2    dx \,dt
\\& \leq C \la  \sum_{\ka'\in R\Z^3; |\ka'-\ka|\leq 2R} \esssup_{0\leq t\leq \la R^2} \int |u|^2\phi(x-\ka')\,dx
\\& \leq C \la  \sum_{\ka'\in R\Z^3; |\ka'-\ka|\leq 2R} e_{R,\la}(\ka').
\end{align*}
For the cubic term,  by Gagliardo-Nirenberg inequality, 
\[
\int_{B_{2R}} |u|^3 dx \lec \bke{\int_{B_{2R}} |\nb u|^2}^{3/4} \bke{\int_{B_{2R}} |u|^2}^{3/4} + R^{-3/2}  \bke{\int_{B_{2R}} |u|^2}^{3/2}.
\]
Thus, denoting $N= \sup_{0\leq t\leq \la R^2}\int_{B_{2R}} |u(t)|^2    \,dx +2\int_0^{\la R^2} \int_{B_{2R}} |\nb u|^2 \,dx\,dt$, we have
\EQ{\label{u3.est}
\int_0^{\la R^2} \int_{B_{2R}} |u|^3 dx \,dt &\lec N^{3/4} \int_0^{\la R^2}\bke{\int_{B_{2R}} |\nb u|^2}^{3/4} dt  + R^{-3/2}  N^{3/2} \la R^2
\\
&\lec N^{3/2} (\la R^2)^{1/4} + N^{3/2} \la R^{1/2}  
\\
&\lec N^{3/2} \la ^{1/4}R^{1/2}
}
using $\la \le 1$.
Thus
we have
\begin{align*}
\int_0^{\la  R^2}\int  |u|^2 (u \cdot \nabla \phi(x-\ka) )\,dx\,ds &\leq \frac C R \int_0^{\la R^2}\int_{B_{2R}( \ka)} |u|^3  \,dx\,ds  
\\&\leq 
C R^{-1/2} \la^{1/4} \sum_{\ka'\in R\Z^3; |\ka'-\ka|\leq 4R} (e_{R,
\la}(\ka'))^{3/2}.
\end{align*}

The only term left is the pressure term.  For it we need to use item 3 from Definition \ref{def:localLeray} to write $p(x,t)$  for $x \in B_{2R}(\ka)$ as 
\EQ{\label{eq.pressure}
p(x,t)&=-\Delta^{-1}\div \div [(u\otimes u )\chi_{4R} (\cdot-\ka)]
\\&\quad - \int_{\R^3} (K(x-y) - K(\ka -y)) (u\otimes u)(y,t)(1-\chi_{4R}(y-\ka))\,dy+c_{x_0,R}(t)
\\&=p_1(x,t)+p_2(x,t) +c_{x_0,R}(t),
}
where $K(x)$, $c_{x_0,R}(t)$ and $\chi_{4R}$ are as in Definition \ref{def:localLeray}.
The benefit of working with this formula is that $K(x-y) - K(\ka -y)$ has extra decay as $|y|\to \I$ when $x$ is close to $\ka$.  In particular
\begin{align}\label{ineq.extradecay}
|K(x-y)-K(\ka-y)|\leq \frac {CR} {| \ka-y |^4},
\end{align}
whenever $|\ka - y|\geq 4R$ and $|x-\ka|\leq 2R$. 
Thus $p_2$ is well-defined even if $u$ has no decay.

For $p_1$, using the Calderon-Zygmund theory we have
\begin{align*}
\|p_1 \|_{L^{3/2} (B_{2R} (\ka))}
&\leq  \| u \chi_{4R}^{1/2}(\cdot-\ka)   \|_{L^3}^2 
\leq  C\sum_{\ka'\in R\Z^3; |\ka'-\ka|\leq 9R}    \| u   \phi^{1/2}(\cdot-\ka')   \|_{L^3}^2 .
\end{align*}
Therefore, using $(\sum_{j=1}^n a_j^2)(\sum_{j=1}^n |a_j|)\le n \sum_{j=1}^n |a_j|^3$ and \eqref{u3.est},
\begin{align*}
\int_0^{\la R^2}\int2 p_1  u\cdot \nabla \phi(x-\ka)  \,dx\,ds
&\leq \frac C R  \int_0^{\la R^2}  \sum_{\ka'\in R\Z^3; |\ka'-\ka|\leq 9R}    \| u   \phi^{1/2}(\cdot-\ka')   \|_{L^3}^3\,ds
\\&\leq C R^{-1/2} \la^{1/4} \sum_{\ka'\in R\Z^3; |\ka'-\ka|\leq 10R} (e_{R,
\la}(\ka'))^{3/2}.
\end{align*}

For $p_2$, we use
the following pointwise estimate for 
$x\in B(\ka, 2R)$,
\begin{align*}
|p_2(x,t)|&\leq C \int \frac R {|\ka-y|^4} u(y,t)^2 (1-\chi_{4R}(y-\ka))\,dy
\\&\leq C\sum_{\ka'\in R\Z^3; |\ka'-\ka|>4R}   \int_{B_{2R}(\ka')} \frac R {|\ka- y|^4 }|u(y,t)|^2 \phi (y-\ka')\,dy
\\&\leq \frac C {R^3} \sum_{  {\ka'}  \in R\Z^3; |\ka'-\ka |>4R}   \frac 1 {|\ka/R-\ka'/R |^4 }  \int_{B_{2R}(\ka')}|u(y,t)|^2 \phi (y-\ka')\,dy
\\&\leq \frac C {R^3} (\overline K*e_{R,\la} )(\ka ),
\end{align*}
where we have used \eqref{ineq.extradecay}, the convolution  $\overline K*e_{R,\la}$ is understood over $R\Z^3$, and, for $x \in R\Z^3$,
\[
\overline K(x) = \frac 1 {|x/R|^4}, \quad \text{if } |x|>4R; \quad 
\overline K(x) =0 \quad \text{otherwise}. 
\] 

We thus obtain, using $\la\leq 1$,
\begin{align*}
&\int_0^{{\la R^2}} \int 2p_2(x,s)  u(x,s)\cdot \nabla \phi (x-\ka) \,dx\,ds 
\\&\leq \frac C R \int_0^{\la R^2} \int_{B_{2R}(\ka)} |p_2|^{3/2}\,dx\,ds + \frac C R \int_0^{\la R^2}\int_{B_{2R}(\ka)}|u|^3\,dx\,ds
\\&\leq C \la^{1/4} R^{-1/2}  ( (\overline K*e_{R,\la} )(\ka))^{3/2} +C \la^{1/4} R^{-1/2} \sum_{|\ka'-\ka|\leq 4R}e_{R,\la}^{3/2}(\ka').
\end{align*}

 We finally note that $\int_0^{\la R^2}\int2 c_{x_0,R} u\cdot \nabla \phi(x-\ka)  \,dx\,ds=0$.

At this point we have established the bound
\EQ{\label{eq3.11}
e_{R,\la}(\ka)&\leq  \int |u_0|^2\phi(x-\ka)  \,dx+C \la  \sum_{\ka'\in R\Z^3; |\ka'-\ka|\leq 2R} e_{R,\la}(\ka')
\\&+ C    \frac {\la^{1/4}} {R^{1/2}}  \sum_{\ka'\in R\Z^3; |\ka'-\ka|\leq 10R} (e_{R,\la}(\ka'))^{3/2}
+C \frac {\la^{1/4}} {R^{1/2}} ( (\overline K*e_{R,\la} )(\ka ))^{3/2},
}
provided $\la \leq 1$.  Note that the constants above do not depend on $q$.
We now raise both sides of the above inequality to the power $q/2$ and sum over $\ka\in R\Z^3$.  The left hand side becomes $E_{R,q,\la}$.   For the first three terms on the right hand side, we have
\[
 \sum_{\ka\in R\Z^3}\bigg(   \int |u_0|^2\phi(x-\ka)  \,dx \bigg)^{q/2} \le C^q A_{0,q}(R)^{q/2},
\] 
\begin{align*}
  \sum_{\ka\in R\Z^3}\bigg(  C \la  \sum_{\ka'\in R\Z^3; |\ka'-\ka|\leq 2R} e_{R,\la}(\ka') \bigg)^{q/2} 
\leq C^q \la^{q/2}   E_{R,q,\la} ,
\end{align*} 
and
\begin{align*}
\sum_{\ka\in R\Z^3}
\bigg( C  \frac {\la^{1/4}} {R^{1/2}} \sum_{\ka\in R \Z^3; |\ka'-\ka|\leq 10R} (e_{R,\la}(\ka'))^{3/2}\bigg)^{q/2}  
&\leq C^q
\bigg( \frac {\la^{1/4}} {R^{1/2}}  \bigg)^{q/2} \sum_{\ka\in R\Z^3} e_{R,\la}(\ka)^{3q/4},
\end{align*}
with $C$ independent of $q$.
Here we have used $(\sum_{i=1}^n a_i )^p \le n^p\sum_{i=1}^n a_i ^p$ for $a_i \ge 0$.

Finally, for the convolution term we use Young's convolution inequality to find
\begin{align*}
\sum_{\ka\in R\Z^3}\bigg(C \frac {\la^{1/4}} {R^{1/2}} ( (\overline K*e_{R,\la} )(\ka))^{3/2}\bigg)^{q/2}
& \leq C^q
\bigg( \frac {\la^{1/4}} {R^{1/2}}  \bigg)^{q/2} \sum_{\ka\in R\Z^3} (\overline K*e_{R,\la}(\ka) )^{3q/4} 
\\
&\leq C^q
\bigg( \frac {\la^{1/4}} {R^{1/2}}  \bigg)^{q/2}   \| \overline K \|_{l^1(R\Z^3)} ^{3q/4} \|  e_{R,\la} \|_{l^{3q/4}}^{3q/4}.
\end{align*}
It is easy to check that $ \| \overline K \|_{l^1(R\Z^3)}$ is bounded independently of $R$.
Now,
since
\[
\|  e_{R,\la} \|_{l^{3q/4}} \le \|  e_{R,\la} \|_{l^{q/2}},
\]
we conclude for $E=E_{R,q,\la}$ and some constant $C_2\ge1$ independent of $q,R$, 
\EQ{\label{Eq.est}
E \leq C_2^q A_{0,q}(R)^{q/2} + C_2^{q} \la^{q/2} E+  C_2^q  \bigg( \frac {\la^{1/4}} {R^{1/2}}  \bigg)^{q/2}  E ^{3/2}.
}
The right side is finite for $\la<R^{-2}T_2$ by assumption \eqref{th2.2-0}.

We claim that $E_{R,q,\la}$ is continuous in $\la$.
Indeed, it is nondecreasing in $\la$ and
\[
E_{R,q,\la} = \sum_{k \in R\Z^3} \bkt{f(k,\la R^2)^{q/2} + g(k,\la R^2)^{q/2} },
\]
where for $k \in R\Z^3$
\EQN{
f(k,t)&=\esssup_{s<t} \int |u(x,s)|^2 \phi(x- k) dx, 
\\
g(k,t)&=\int_0^t \int  |\nb u(x,s)|^2 \phi(x- k) dx\,ds.
}
They are both nondecreasing.
We first show the continuity of $f(k,t)$ in $t$ for fixed $k$. (The continuity of $g(k,t)$ in $t$ is clear.)
For $0<h \ll 1$, choose $\th(s) \in C^1_c(\R_+)$ such that $\th(s) = 1$ for $s \in [t,t+h]$, $\th(s) = 0$ for $s<t-h$, and $0\le \th'(s)\le 2/h$ for $s \in [t-h,t]$. By local energy inequality \eqref{CKN-LEI} with test function $\th(s) \phi(x-k)$, 
\EQN{
\esssup _{s \in [t,t+h]}  \int  |u(x,s)|^2 \phi(x-k) dx 
&\le  \int_{t-h} ^{t} |u|^2\th'(s) \phi(x-k)ds + \int_{t-h} ^{t+h} \si(s)ds,
}
where
\EQN{
\si(s)= C\int_{B(k,2R)} \bke{R^{-2} |u|^2+ R^{-1} |u|^3+ R^{-1} |p-c_{x_0,R}|^{3/2} } (x,s) \, dx
}
is integrable.
Thus
\[
f(k,t+h) \le  f(k,t)  \int_{t-h} ^{t} \th'(s)ds + \int_{t-h} ^{t+h} \si(s)=f(k,t) + \int_{t-h} ^{t+h} \si(s).
\]
This shows the continuity of $f(k,t)$ in $t$ for fixed $k$. We now show the continuity of $E_{R,q,\la}$ in $\la$.
For any $\e>0$, there is $N>1$ such that
\[
\sum_{k \in R\Z^3; |k|<N} \bkt{f(k,\la R^2)^{q/2} + g(k,\la R^2)^{q/2} } > E_{R,q,\la} -\frac{\e}2.
\]
Since the left side is a finite sum and each summand is nondecreasing and continuous in $t$,
there is $\tau<\la $ such that 
\[
\sum_{k \in R\Z^3; |k|<N} \bkt{f(k,\tau R^2)^{\frac q2} + g(k,\tau R^2)^{\frac q2} } 
> \sum_{k \in R\Z^3; |k|<N} \bkt{f(k,\la R^2)^{\frac q2} + g(k,\la R^2)^{\frac q2} }  -\frac{\e}2.
\]
Hence $E_{R,q,\la}\ge E_{R,q,\tau} > E_{R,q,\la}  -\e$. This shows 
the continuity of $E_{R,q,\la}$ in $\la$.

Since $E_{R,q,\la}$ is continuous in $\la$,  from \eqref{Eq.est} we conclude
\[
E \leq 2E_0, \quad E_0 = C_2^q (A_{0,q}(R))^{q/2} ,
\]
if $C_2^{q} \la^{q/2}\le 1/4$ and $C_2^q  \bigg( \frac {\la^{1/4}} {R^{1/2}}  \bigg)^{q/2}  (2E_0) ^{1/2} \le 1/4$, which is achieved if (using $q\ge 2$)
\EQ{
\la \le  {\la_R:=} \min (\la_0, \frac{\la_0 R^2}{A_{0,q}(R)^{2}}),
}
where  {$\la_0 = \min( (2C_2)^{-2},(2C_2)^{-12})$.}  This shows the first estimate  \eqref{th2.2-1} of Lemma \ref{lem.A0qbound} with $C_1 = CC_2^2$. Note that the constants $C_2,\la_0$ and $C_1$ do not depend on $q$ and $R$.

We now show \eqref{th2.2-2}. By Gagliardo-Nirenberg inequality, 
\[
\int_{B_{R}} |u|^{\frac {10}3} dx \lec \bke{\int_{B_{R}} |\nb u|^2} \bke{\int_{B_{R}} |u|^2}^{2/3} + R^{-2}  \bke{\int_{B_{R}} |u|^2}^{5/3}.
\]
Denoting $N= \sup_{0\leq t\leq \la R^2}\int_{B_{R}} |u(t)|^2    \,dx +2\int_0^{\la R^2} \int_{B_{R}} |\nb u|^2 \,dx\,dt$ with $\la=\la_R$, we have
\EQ{\label{u103.est}
\int_0^{\la R^2} \int_{B_{R}} |u|^{\frac {10}3} dx \,dt 
&\lec N^{2/3} \int_0^{\la R^2}\bke{\int_{B_{R}} |\nb u|^2} dt  + R^{-2}  N^{5/3} \la R^2
\\
&\lec N^{5/3}  + \la N^{5/3}  \lec    N^{5/3} 
}
using $\la \le 1$.
For $k \in R\Z^3$ and $Q(k) = B_R(k) \times (0,\la_R R^2)$, by \eqref{u103.est} with $B_R$ replaced by $B_R(k)$,
we have $N \le e_{R,\la}(k)$ and hence
\EQ{
\sum_{k \in R\Z^3} \bke{\int_{Q(k)} |u|^{\frac {10}3} dx\,dt}^{\frac {3q}{10}} 
\le C \sum_{k \in R\Z^3}
(  e_{R,\la}(k)^{\frac 53} )^{\frac {3q}{10}} \le C   E_0.
}
Thus
\[
\sum_{k \in R\Z^3} \bke{\int_{Q(k)} |p_1|^{\frac 53} dx\,dt}^{\frac {3q}{10}} 
\le C\sum_{k \in R\Z^3}\sum_{k' \in R\Z^3; |k-k'|<10R}\bke{\int_{Q(k')} |u|^{\frac {10}3} dx\,dt}^{\frac {3q}{10}} 
\le C E_0.
\]
For $p_2$, recall $p_2$ in $B_R(k)$ is bounded by $R^{-3} \bar K * e_{R,\la}(k)$ and hence
\[
\int_{Q(k)} |p_2|^{\frac53} dx\,dt \le C\la  ( \bar K * e_{R,\la}(k))^{\frac 53}.
\]
Thus
\EQN{
\sum_{k \in R\Z^3} \bke{\int_{Q(k)} |p_2|^{\frac 53} dx\,dt}^{\frac {3q}{10}} 
&\le  C\la^{\frac {3q}{10}}  \sum_{k \in R\Z^3}( \bar K * e_{R,\la}(k))^{\frac q2}
\\
&\le   C\la^{\frac {3q}{10}}  \norm{ \bar K }_{ l^1}^{\frac q2} \sum_{k \in R\Z^3}( e_{R,\la}(k))^{\frac q2}
\le   C   E_0.
}
We conclude
\[
\bigg\|  \int_{Q(k)} |u|^{\frac {10}3} +|p_1+p_2|^{\frac53}\,dx\,dt \bigg\|_{l^{\frac {3q}{10}}(k \in R\Z^3)}
\le C  E_0^{\frac {10}{3q}}  
=  C  A_{0,q}(R)^{\frac 53}.
\]
This shows \eqref{th2.2-2} and completes the proof.
 \end{proof}
 
 \begin{remark}
To prove $E_{R,q,\la}\le 2E_0$, instead of proving \eqref{Eq.est} and continuity of $E_{R,q,\la}$ in $\la$, we may prove an integral inequality of the form $E_\la \le E_0 + C \int_0^\la( E_s + E_s^{3/2})\, ds$ and then apply Gronwall inequality  \cite[Lemma 2.2]{BT8} for discontinuous functions.
 
 \end{remark}

We now prove
Theorem \ref{thrm.boundE2q}, which is an easy corollary of Lemma \ref{lem.A0qbound}.

\begin{proof}[Proof of Theorem \ref{thrm.boundE2q}]
  
Note that, by  $R\ge1$, \eqref{NqRbound}, and H\"older inequality,
\EQ{\label{ineq.NqR0}
N_{q,R}^0(u_0) & 
\le \frac C R   \bigg(  \sum_{k\in \Z^3} \bigg(  \sum_{|i-kR|<R} a_i^2 \bigg)^{q/2} 			\bigg)^{2/q}
\\&\leq  \frac C R   \bigg(  \sum_{k\in \Z^3}   \sum_{|i-kR|<R} a_i^q  R^{(3-6/q)q/2}\bigg)^{2/q}
\\&\leq C R^{2-6/q}\|u_0\|_{E^2_q}^2,
}
where $a_i = ( \int_{B_1(i)}|u_0|^2\,dx )^{1/2}$ and $i\in \Z^3$.
Applying Lemma \ref{lem.A0qbound}, we see that
\[
\bigg\|\esssup_{0\leq t \leq \la_R R^2} \int_{B_R(x_0R) }\frac {|u|^2}2  \,dx+ \int_0^{\la_R R^2}\int_{B_R(x_0R) } |\nabla u|^2\,dx\,dt \bigg\| _{l^{q/2}(x_0\in \Z^3)}
\leq  C_1 A_{0,q}(R).
\]
Considering the definition of $\la_R$ in the statement of Lemma \ref{lem.A0qbound}, and reciprocating the upper bound \eqref{ineq.NqR0}, we have
\[ 
\la_R R^2=\min (\la_0 R^2,\,  \frac{\la_0 R^2}{ (N_{q,R}^0(u_0))^{2} }) \geq \min (\la_0 R^2,\, \frac{\la_0 R^{12/q -2}}{C^2 \norm{u_0}_{E^2_q}^4})
\ge  \frac{ \la_1 R^{\min (2, 12/q -2)}} { (1+ \norm{u_0}_{E^2_q})^{4} }
\]
where $\la_1 = \la_0 (1+C)^{-2}$.
Also, $A_{0,q}(R)=RN^0_{q,R}(u_0) \leq CR^{3-6/q}\|u_0\|_{E^2_q}^2$ by \eqref{ineq.NqR0}. This gives the upper bound in the statement of Theorem \ref{thrm.boundE2q}.
\end{proof}

\section{Local existence of solutions in $\LE_q$}\label{sec.local}
 
In this section we construct local solutions in $\LE_q$ for initial data in $E^2_q$ up to a time  identical to that given in Lemma \ref{lem.A0qbound} for $R=1$, namely 
$T \sim \lambda_0\min(1, \|u_0\|_{E^2_q}^{-4} )$.  Doing so requires studying a perturbed version of the Navier-Stokes equations, namely
\Eq{\label{eq4.1}
\partial_t u -\Delta u +u\cdot\nb u +v\cdot \nb u +u\cdot\nb v+\nb p = 0,\qquad \nb\cdot u=0,
}
where $v$ is a given divergence free vector field.  We begin by analyzing the perturbed problem and then return to \eqref{eq.NSE}. Note that the overall approach we follow to constructing solutions in spaces larger than $L^2$ is due to Lemari\'e-Rieusset \cite{LR}.

\subsection{Existence for the regularized, perturbed problem}

We will solve a regularization of the perturbed system for $L^2$ data and then use this to construct a solution with $E^2_q$ data.
The regularized, perturbed problem is well understood when $u_0\in L^2$, see \cite[Ch 21 \S3]{LR}.  However, we need to start from the level of the Picard iterates to establish $E^2_q$ bounds.

\begin{lemma}\label{lemma.existence.loc.reg}
Let  $\e\in (0,1]$ and $\delta>0$ be given, $T_0>0$ and $\eta$ be a spatial mollifier in $\R^3$.
Assume $u_0\in L^2(\R^3)$ and is divergence free, and $v:\R^3\times [0,T_0]\to \R^3$ satisfies 
$\div v=0$ and
\[
\esssup_{0<t\leq T_0}\| v(t) \|_{L^3_\uloc}
<\delta.
\]
Then, there exists $ T_{\e,\delta} =\min \big(T_0, C(\e)(\|u_0\|_{E^2_q}+ \de)^{-2}\big)$ and a mild solution $u_\e$ of the integral equation 
\Eq{\label{eq.mildregularized}
u_\e(x,t) = e^{t\Delta} u_0(x) +\int_0^te^{(t-s)\De}\mathbb P \nb\cdot ((\eta_\e *u_\e )\otimes u_\e )\,ds +L_t(u_\e),
}
for $0<t<T_{\e,\de}$, where 
\[
L_t(u_\e) = \int_0^t e^{(t-s)\Delta} \mathbb P \cdot\nb ((\eta_\e* v)\otimes u_\e +u_\e \otimes (\eta_\e* v)  )\,ds,
\] 
with $u_\e\in \LE_q (0,T_{\e,\de}) \cap C([0,T_{\e,\de}];L^2)$, and $u^\e$ satisfies
\EQ{\label{ineq.target}
\bigg\| \bigg(\esssup_{0<t<T_{\e,\delta} } \int_{B_1(k)} |u_\e(x,t)|^2\,dx \bigg)^{\frac12}\bigg\|_{l^{q}}\leq 2C \|u_0\|_{E^2_q}
\text{ and }\sup_{0<t<T_{\e,\delta} } \|u(t)\|_{L^2 }\leq 2C \|u_0\|_{L^2},
}
for a universal 
constant $C$. It is the unique mild solution of \eqref{eq.mildregularized} in the class \eqref{ineq.target}.
There exists a pressure $p_\e$ so that $u_\e$ and $p_\e$ solve
\[
\partial_t u_\e -\Delta u_\e +(\eta_\e* u_\e) \cdot\nb u_\e + (\eta_\e * v )\cdot \nb u_\e + u_\e \cdot \nb (\eta_\e *v) +\nb p_\e = 0;\quad \nb\cdot u_\e=0,
\]
in the weak sense on $\R^3\times (0,T_{\e,\de})$. Finally, $u_\e$ and $p_\e$ are smooth by the interior regularity of the Stokes equations with smooth coefficients.
\end{lemma}

We do not assume $\e,\de \ll 1$ in Lemma \ref{lemma.existence.loc.reg}.
Note that $T_{\e,\de}$ depends on $\e$, $\de$ and $u_0$ (through $\norm{u_0}_{E^2_q}$). Also, we do not establish estimates on the gradients of the velocity yet.  This is the same mollified perturbed Navier-Stokes equations considered in \cite[Ch 21 \S3]{LR} (which contains an additional temporal mollification on $v$ that can be ignored)
but we are asserting additional properties, namely bounds in $E^2_q$. Due to this, we get for free properties like $u_\e \in C([0,T_{\e,\de}];L^2)$ from Lemari\'e-Rieusset's proof.

\begin{proof} Note that $u_0\in E^2_q$ since $L^2\subset E^2_q$.  
We establish estimates first for elements of the Picard scheme where
\[
u_\e^1 = e^{t\Delta }u_0,
\]
and, for $n>1$,
\EQ{
u_\e^n(x,t)&= e^{t\Delta }u_0(x) +\int_0^t e^{(t-s)\Delta}\mathbb P\nb\cdot ( (\eta_\e*u_\e^{n-1})\otimes u_\e^{n-1}    )\,ds
 + L_t(u_\e^{n-1}).
}

For the first iterate we have
\EQN{
\int_{B_1(k)} |e^{t\Delta} u_0 (x)|^2\,dx &\leq  \int_{B_1(k)} |e^{t\Delta} ( u_0  \chi_{B_4(k)}) (x)|^2\,dx 
\\&+ \int_{B_1(k)} \sum_{|k'-k|\geq 4}  \int_{B_1(k')} \bigg| \frac {|G(|x-y|^2/t)|} {t^{3/2}} u_0(y)\bigg|^2\,dy\,dx = I_1^0(k)+I_2^0(k),
}
where $G(u)=Ce^{-u/4}$.  By $L^p$ estimates for the heat kernel we have
\Eq{\label{ineq.picard.1.a}
\int_{B_1(k)} |e^{t\Delta} ( u_0  \chi_{B_4(k)}) (x)|^2\,dx \leq C \sum_{|k-k'|\leq 4} \int_{B_1(k')}|u_0(x)|^2\,dx.
}
For the second term, since $|G(u)|\leq C |u|^{-2}$,
\EQ{\label{ineq.picard.1.b}
I_2^0(k) &\leq C\int_{B_1(k)} t  \bigg[\sum_{|k'-k|\geq 4}  \int_{B_1(k')} \frac 1 {|x-y|^4} |u_0(y)|\,dy \bigg]^2\,dx
\\&\leq  C\int_{B_1(k)} t  \bigg[\sum_{|k'-k|\geq 4}  \frac 1 {|k'-k|^4}\bigg(\int_{B_1(k')} |u_0(y)|^2\,dy\bigg)^{1/2}\bigg]^{2}\,dx
\leq  C  ( \widetilde K * a) (k)^2,
}
where we have assumed $t\leq 1$ and set $a=\{  a_k\}_{k\in \Z^3}$ is the sequence with entries
\[
a_k =\bigg( \int_{B_1(k)} |u_0(y)|^2\,dy\bigg)^{1/2},
\]
and 
\[
\widetilde K_k = \begin{cases}
 |k|^{-4}  &\text{ if }|k|\geq 4
\\ 0 &\text{ otherwise}
\end{cases}   .
\]
All terms above are independent of $t$.  Therefore,
\EQ{
\bigg\| \bigg(  \sup_{0<t<1} \int_{B_1(k)} |u_\e^1(x,t)|^2\,dx \bigg)^{1/2}  \bigg\|_{l^{q}}
&\leq C \|u_0\|_{E^2_q} +C  \|    \widetilde K * a \|_{l^{q}}
\\&\leq C \|u_0\|_{E^2_q} +C  \|   a   \|_{l^{q}}  \| \widetilde K\|_{l^{1}}
=  C \|u_0\|_{E^2_q}.
}

For the higher Picard iterates we need to use standard estimates for the Oseen tensor $S$, the kernel of $e^{t\Delta}\mathbb P$ in $\R^3$, found by Oseen \cite{Oseen}. We have the following pointwise estimate for $S$ by Solonnikov \cite{Solonnikov},
\Eq{\label{ineq.oseen}
|\pd_t^m\nb_x^k S(x,t)|\le \frac{C_{k,m}}{(|x|+\sqrt t)^{3+k+2m}}.
}

We are now ready to estimate the $n$-th Picard iterate using the assumption that
\Eq{\label{eq4.9}
\bigg\|    \sup_{0<t<T_\e} \bigg(\int_{B_1(k)}|u_\e^{n-1} (x,t) |^2\,dx\bigg)^{1/2}		\bigg\|_{l^{q}} < 2C\|u_0\|_{E^2_q}.
}
We have 
\Eq{\label{eq4.10}
\int_{B_1(k)} |u_\e^n(x,t)|^2\,dx \leq  \int_{B_1(k)} |e^{t\Delta} u_0 (x)|^2\,dx + I(k)+J(k),
}
where
\EQN{
I(k)&= \int_{B_1(k)}\bigg| \int_0^t e^{(t-s)\De}\mathbb P \nb\cdot (\eta_\e*u_\e^{n-1}   \otimes u_\e^{n-1}) \,ds\bigg|^2\,dx,
\\
J(k)&=\int_{B_1(k)} \big|  L_t(u_\e^{n-1})\big|^2\,dx.
}
We've already estimated the first term on the right hand side of \eqref{eq4.10}.  For $I(k)$, using Cauchy-Schwarz inequality,
\EQN{
I(k)&\le \int_{B_1(k)}\bigg| \int_0^t e^{(t-s)\De}\mathbb P \nb\cdot (\eta_\e*u_\e^{n-1}   \otimes u_\e^{n-1}  \chi_{B_{4}(k)^c})  \,ds\bigg|^2\,dx\\
&\quad+ \int_{B_1(k)}\bigg| \int_0^t e^{(t-s)\De}\mathbb P \nb\cdot (\eta_\e*u_\e^{n-1}   \otimes u_\e^{n-1}  \chi_{B_{4}(k)})  \,ds\bigg|^2\,dx
\\
&\le C t \int_0^t  \int_{B_1(k)}\bigg| e^{(t-s)\De}\mathbb P \nb\cdot (\eta_\e*u_\e^{n-1}   \otimes u_\e^{n-1}\chi_{B_{4}(k)^c})\bigg|^2 \,dx \,ds  \\
&\quad + C \bigg\|  \int_0^t e^{(t-s)\De}\mathbb P \nb\cdot (\eta_\e*u_\e^{n-1}   \otimes u_\e^{n-1}  \chi_{B_{4}(k)})  \,ds \bigg\|_{L^2_\uloc}^2\\
&= : I_1(k)+ I_2(k).
}

Using the pointwise estimates \eqref{ineq.oseen} for the kernel of $e^{t\Delta}\mathbb P \nb\cdot$ we have
\begin{align}
I_1(k) &\leq  C t \int_0^t  \int_{B_1(k)} \bigg[ \sum_{|k-k'|>4}  \int_{B_1(k')} \frac 1 {|k-k'|^4} |\eta_\e*u_\e^{n-1}   \otimes u_\e^{n-1}| \,dy  \bigg]^2 \,dx \,ds \notag 
\\&\leq 
C(\e) t  \sup_{0<s<t}\|u_\e^{n-1}\|_{E^2_q}^2 \int_0^t  \int_{B_1(k)} \bigg[ \sum_{|k-k'|>4}  \frac 1 {|k-k'|^4}  \bigg(\sup_{0<s<t} \int_{B_1(k')}|  u_\e^{n-1}|^2 \,dy\bigg)^{\frac12}  \bigg]^2 \,dx \,ds \notag 
\\&\leq C(\e) t^2  \big( \sup_{0<s<t} \|u_\e^{n-1}\|_{E^2_q}^2  \big) (\widetilde K * a^{n-1}) (k)^2  , \label{ineq.i1}
\end{align}
where we have used the fact that 
that $\|\eta_\e * u_\e^{n-1}\|_{L^\infty}\leq C(\e) \|u_\e^{n-1}\|_{E^2_q}$ and are letting $a^{n-1}$ be the sequence with entries \[a_k^{n-1}=  \bigg( \sup_{0<s<t} \int_{B_1(k)}|u_\e^{n-1}(y)|^2\,dy \bigg)^{1/2}.\]

For $I_2(k)$, using the estimates \cite[(3.20) on p.~385]{MaTe} with $p=q=2$ and $t<1$,
\EQ{\label{ineq.i2}
I_2(k) &\leq \bigg( \int_0^t \frac C {(t-s)^{1/2}} \| \eta_\e*u_\e^{n-1}   \otimes u_\e^{n-1}  \chi_{B_{4}(k)} \|_{L^2_\uloc}   		\,ds		\bigg)^2
\\&\leq \bigg(C(\e)\int_0^t \frac 1 {(t-s)^{1/2}}  \sum_{|k-k'|<8} \|u_\e(s)\|_{L^2(B_1(k'))}^2\,ds  \bigg)^2
\\&\leq    C(\e) t  \sup_{0<s<t} \|u_\e^{n-1}\|_{E^2_q}^2   \sup_{0<s<t} \sum_{|k-k'|<8}  \|u_\e^{n-1}\|_{L^2(B_1(k'))}^2.
}

We now turn our attention to $J(k)$.  We use the same strategy as our estimate for the nonlinear term. In particular,
\[
J(k)=\int_{B_1(k)} \big|  L_t(u_\e^{n-1}) \big|^2\,dx \leq J_1(k)+J_2(k),
\]
where
\EQN{
J_1(k)&=C t \int_0^t  \int_{B_1(k)}\bigg| e^{(t-s)\De}\mathbb P \nb\cdot    (   ( u_\e^{n-1}   \otimes \eta_\e *v+\eta_\e* v\otimes u_\e^{n-1}   )\chi_{B_{4}(k)^c})\bigg|^2 \,dx \,ds,\\
J_2(k)&=C\bigg\|  \int_0^t e^{(t-s)\De}\mathbb P \nb\cdot (  (  u_\e^{n-1}   \otimes\bar  \eta_\e*v+\eta_\e* v\otimes u_\e^{n-1}   )\chi_{B_{4}(k)})  \,ds \bigg\|_{L^2_\uloc}^2.
}
 
To estimate $J_1(k)$, note that  its integrand is bounded by
\EQN{
&\bigg[ \sum_{|k-k'|>4}  \int_{B_1(k')} \frac 1 {|k-k'|^4} | ( u_\e^{n-1}   \otimes \bar  \eta_\e*v+\eta_\e* v\otimes u_\e^{n-1}   )| \,dy  \bigg]^2(s)
\\&\leq C \bigg[ \sum_{|k-k'|>4} \frac 1 {|k-k'|^4} \big( \| u_\e^{n-1} \|_{L^2(B_2(k') )} \|v\|_{L^\infty(0,T_0;L^3_\uloc)}   \big)    \bigg]^2(s)
\\&\leq C \delta^2  \big(   \widetilde K * a^{n-1} \big)^2(k),
}
implying
\EQ{\label{ineq.j1}
J_1(k)\leq C \delta^2 t^2     \big(\widetilde K * a^{n-1} \big)^2(k) .
}
On the other hand, 
\EQ{\label{ineq.j2}
J_2(k)&\leq \bigg( \int_0^t \frac C {(t-s)^{1/2}} \| (  u_\e^{n-1}   \otimes \eta_\e* v+\eta_\e* v\otimes u_\e^{n-1}   )\chi_{B_{4}(k)} \|_{L^2_\uloc}   		\,ds		\bigg)^2
\\&\leq \bigg(   \int_0^t \frac C {(t-s)^{1/2}} \|\eta_\e*v \|_{L^\infty (0,T_0; L^\infty  (B_4(k)))} \|u_\e^{n-1}\|_{L^2(B_4(k))}  \,ds\bigg)^2
\\&\leq \bigg(   \int_0^t \frac {C}{(t-s)^{1/2}} C(\e)  \| v \|_{L^\I(0,T_0;L^3_\uloc)} \|u_\e^{n-1}\|_{L^2(B_4(k))} \,ds
 \bigg)^2
\\&\leq   C(\e) t \delta^2   \| u_\e^{n-1} \|_{L^2 (B_4(k))}^2.
}

Combining \eqref{ineq.picard.1.a}, \eqref{ineq.picard.1.b}, \eqref{ineq.i1}, \eqref{ineq.i2}, \eqref{ineq.j1} and \eqref{ineq.j2}, we have
\EQ{\label{ineq.combined}
\int_{B_1(k)} |u_\e^n(x,t)|^2\,dx &\leq  C \sum_{|k-k'|\leq 4} \int_{B_1(k')}|u_0(x)|^2\,dx+ C   (\widetilde K * a)(k)^2
\\&+C(\e) t^2  \big( \sup_{0<s<t}  \|u_\e^{n-1}(s)\|_{E^2_q}^2 \big) ( \widetilde K * a^{n-1}) ^2( k)
\\&+ C(\e) t  \sup_{0<s<t} \|u_\e^{n-1}(s)\|_{E^2_q}^2   \sup_{0<s<t} \sum_{|k-k'|<8}  \|u_\e^{n-1}\|_{L^2(B_1(k'))}^2
\\&+ C\de^2 t^2     \big(\widetilde K * a^{n-1} \big)^2(k)+ C(\e)\de^2 t    \| u_\e^{n-1} \|_{L^2 (B_4(k))}^2.
}
Note that  
\[
\sup_{0<s<t}\|f(s)\|_{E^2_q} \leq \bigg\|  \sup_{0<s<t} \bigg(\int_{B_1(k)} |f(x,s)|^2\,dx \bigg)^{1/2}  \bigg\|_{l^q}.
\]
Taking the supremum in time of the left hand side of \eqref{ineq.combined}, applying the $l^{q/2}$ norm, using Young's convolution inequality, and raising everything to the $1/2$ power yields 
\EQ{\label{eq4.16}
\norm{u_\e^n}_{\LE_q^\flat(0,t)} 
&\leq C\|u_0\|_{E^2_q} +C(\e) t^{1/2} \norm{u_\e^{n-1}}_{\LE_q^\flat(0,t)} ^2
\\&\quad +C(\e) \de t^{1/2}  \norm{u_\e^{n-1}}_{\LE_q^\flat(0,t)}.
}
Recall \eqref{LEqflat.def} for the $\LE_q^\flat$-norm.
So, if $t$ is small as determined by $C(\e)$, $\delta$ and $\|u_0\|_{E^2_q}$
  (but independently of $n$),
\[
 t \le \frac{C(\e)} {\|u_0\|_{E^2_q}^2+ \de^2},
\] then the right hand side of \eqref{eq4.16} is controlled by $2C\|u_0\|_{E^2_q}$.
This establishes a uniform in $n$ bound for $u_\e^{n}$.

Proceeding in the standard way,  these uniform bounds and the estimation methods above allow us to show the difference estimate
\Eqn{
 \norm{u_\e^{n+1}-u_\e^{n}}_{\LE_q^\flat(0,t)}
\le C(\e) \sqrt t (\norm{u_0}_{E^2_q} + \de) 
\norm{u_\e^{n}-u_\e^{n-1}}_{\LE_q^\flat(0,t)}.
}
Thus, if $t$ is sufficiently small, $u_\e^{n}$ is a Cauchy sequence in the above norm and converges to a limit $u_\e$ in the sense that 
\[
\norm{u_\e^{n}-u_\e}_{\LE_q^\flat(0,t)}
\to 0, \quad
\text{as } n \to \I.
\]
This convergence implies $u^\e$  satisfies \eqref{eq.mildregularized}
and \eqref{ineq.target}.

Uniqueness in the class \eqref{ineq.target} follows from the same difference estimates: If $u_1$ and $u_2$ are two mild solutions of \eqref{eq.mildregularized}
satisfying \eqref{ineq.target}, then
\Eqn{
 \norm{u_1-u_2}_{\LE_q^\flat(0,t)}
\le C(\e) \sqrt t (\norm{u_0}_{E^2_q} + \de) 
\norm{u_1-u_2}_{\LE_q^\flat(0,t)}.
}
Thus $u_1=u_2$ if $t$ is sufficiently small.

We now recover a pressure $p_\e$ associated to $u_\e$. It is known (see e.g.~\cite[Ch 21]{LR}) that $u_\e\in L^\I(0,T_{\e,\de};L^2 )$ for some $T_{\e,\de}'>0$ and its size is bounded by $2C \|u_0\|_{L^2}$ (this may be smaller than $T_{\e,\de}$ in which case we redefine $T_{\e,\de}$ to be the smaller of these values).  Hence, 
\[
\eta_\e * u_\e \otimes u+\eta_\e *v \otimes u_\e +u_\e \otimes \eta_\e * v\in L^\I(0,T_{\e,\de};L^2 ),
\]
and, therefore, $p_\e = (-\Delta)^{-1}\partial_i\partial_j  (\eta_\e * u_\e \otimes u+\eta_\e *v \otimes u_\e +u_\e \otimes \eta_\e * v)$ is meaningful.  It follows that $u_\e-e^{t\Delta}u_0$ solves the Stokes system with pressure $p_\e$ and source term $\nb \cdot (\eta_\e * u_\e \otimes u+\eta_\e *v \otimes u_\e +u_\e \otimes \eta_\e * v)$. Adding $e^{t\Delta}u_0$ shows that $u_\e$ and $p_\e$ solve the perturbed, regularized Navier-Stokes equations. Note that the local pressure expansion \eqref{pressure.dec} follows from the definition of $p_\e$.

We finally show that
\Eq{\label{eq4.18}
\bigg\| \int_0^t\int_{B_1(k)} |\nb u_\e|^2\,dx\,ds \bigg\|_{l^{q/2}(k)} <\infty,
}  
as this is needed to have $u_\e\in \LE_q(0,T_{\e,\de})$.
The following local energy equality holds for $u^\e$ and $p^\e$ due to smoothness and convergence to the data in $L^2_\loc$ for $t\leq T_{\e,\delta}$:
\EQ{
&\int_{B_1(k)} |u_\e|^2(x,t) \phi(x-k)\,dx + 2\int_0^t\int |\nb u_\e|^2 \phi(x-k)\,dx\,ds  
\\&= \int |u_0|^2\phi(x-k)  \,dx+ \int_0^t\int |u_\e|^2 \Delta\phi(x-k)  \,dx\,ds 
\\&\quad+\int_0^t\int  |u_\e|^2 ( \eta_\e*u_\e + \eta_\e *v)\cdot \nb \phi(x-k)\,dx\,ds
\\&\quad-  2 \int_0^t \int (u_\e \cdot \nb  (\eta_\e  *v ) )\cdot u_\e \phi(x-k)\,dx\,ds
+\int_0^t\int2 p_\e (u_\e\cdot \nabla\phi(x-k) )\,dx\,ds.
}
Because $\|\eta_\e * u\|_{L^\infty(B_2(k))}\leq C(\e) \|u_\e\|_{L^2_\uloc}$ provided $\e\le1$, we have
\Eq{\label{ineq.locbound1}
\int_0^t\int  |u_\e|^2( \eta_\e*u_\e)\cdot \nb \phi(x-k)\,dx\,ds\leq C(\e) \|u_\e\|_{L^\infty L^2_\uloc} \esssup_{0<s<t}\sum_{k'\sim k} \int_{B_1(k')} |u_\e|^2\,dx.
}
Using our assumptions on $v$  we have 
\[
 \|  \eta_\e * v\|_{L^\infty}+
\| \nb (\eta_\e * v)\|_{L^\infty}\leq C(\e) \|v\|_{L^\infty(0,T_0; L^{3}_\uloc)}\le C(\e)\de. 
\]
Hence 
\Eq{\label{ineq.locbound2}
\int_0^t\int  |u_\e|^2  (\eta_\e *v)\cdot \nb \phi(x-k))\,dx\,ds 
\leq  C(\e) \delta \esssup_{0<s<t} \sum_{k'\sim k}  \int_{B_1(k')} |u_\e(x,s)|^2\,dx,
}
and
\Eq{\label{ineq.locbound3}
\int_0^t  \int ((u_\e \cdot \nb) \eta_\e * v)\cdot (u_\e \phi(x-k)) \,dx\,ds \leq 
C(\e) \delta \esssup_{0<s<t}\sum_{k'\sim k} \int_{B_1(k')} |u_\e|^2\,dx.
}

The pressure satisfies the local pressure expansion \eqref{pressure.dec} and, therefore, by introducing a constant we may write it as a sum of a Calderon-Zygmund operator applied to a localized term and a non-singular integral operator applied to a far-field term, that is $p_\e(x,t) +c = p_{\e,\near }+p_{\e,\far}$. Given the structure of the pressure integral in the local energy inequality, the constant plays no role.  After applying the Calderon-Zygmund inequality, the term involving $p_{\e,\near }$ can be bounded exactly as the nonlinear and perturbation terms above, namely this term in the local energy inequality is dominated by the sum of the right hand sides of \eqref{ineq.locbound1}-\eqref{ineq.locbound3}.
We therefore only have to estimate the far field part of the pressure.  As usual, in $B_2(k)\times (0,T_0)$,
\EQN{
|p_{\e,\far}|&\leq C\sum_{k'\in \Z^3; |k'-k|>4} \frac 1 {|k-k'|^4} \int_{B_2(k')} \big(|u_\e| |\eta_\e * u_\e|  + |\eta_\e * v| |u_\e|         \big)\,dy
\\&\leq C(\e)   \sum_{k'\in \Z^3; |k'-k|>4} \frac 1 {|k-k'|^4} \bigg(\|u_\e\|_{L^2(B_3(k'))}^2   + \|u_\e\|_{L^2(B_3(k'))}\|v\|_{L^\infty(0,T_0;L^3_\uloc) }     \bigg),
}
Therefore,
\EQ{\label{eq4.23}
&\int_0^t \int_{B_2(k)}2  {p_{\e,\far}} (u_\e\cdot \nabla\phi(x-k) )\,dx\,ds
\\&\leq    C(\e) T_0 \|u_\e\|_{L^\infty L^2_\uloc} \esssup_{0<s<t}\sum_{k'\in \Z^3; |k'-k|>4} \frac 1 {|k-k'|^4}   \int_{B_3(k')} |u_\e|^2  \,dy  
\\
&+  C(\e)\delta T_0 \|u_\e\|_{L^\infty L^2(B_2(k))} \esssup_{0<s<t}\sum_{k'\in \Z^3; |k'-k|>4} \frac 1 {|k-k'|^4}    \|u_\e\|_{L^2(B_3(k'))}	
}

Taking the essential supremum in $t$, raising the above to the power $q/2$, summing over $k$ and using H\"older's and  Young's inequalities on the far field pressure term shows \eqref{eq4.18}. Indeed, if we denote $\alpha _k = \|u_\e\|_{L^\infty L^2(B_2(k))} $ with $\alpha =(\alpha _k) \in l^q$, then  terms from \eqref{eq4.23} are bounded by
\[
\norm{\widetilde K * (\alpha ^2)}_{l^{q/2}} + \norm{\alpha _k (\widetilde K * \alpha )(k)}_{l^{q/2}}\lec 
\norm{\alpha ^2}_{l^{q/2}} +  \norm{\alpha }_{l^{q}} \norm{\widetilde K * \alpha}_{l^{q}}\lec \norm{\alpha }_{l^{q}}^2.
\]
This finishes the proof of Lemma \ref{lemma.existence.loc.reg}.
\end{proof}

\subsection{A priori bound for the perturbed problem}

A local energy solution to the perturbed Navier-Stokes equations \eqref{eq4.1} is a weak solution $u$ to \eqref{eq4.1} satisfying Definition \ref{def:localLeray} with the obvious modifications, namely $u$ and $p$ satisfy the perturbed system as distributions and also satisfy the perturbed local energy inequality. Lemma \ref{lem.A0qbound} can be extended to hold for local energy solutions to the perturbed Navier-Stokes equations provided the perturbation is chosen appropriately.

\begin{lemma}\label{lem.A0qboundperturbed}
Assume $u_0\in E^2_q$ for some $2\leq q<\infty$ is divergence free.  There exists a small universal constant $c_0$ so that, if $\delta \in (0,c_0]$ and $v:\R^3\times [0,T_0]\to \R^3$ satisfies
$\div v=0$ and
\[
\esssup_{0<t\leq T_0}\| v(t) \|_{L^3_\uloc} <\delta,
\]
for some $T_0>0$ and, additionally,  a given local energy solution $u$ to the perturbed Navier-Stokes equations \eqref{eq4.1} satisfies,
 \begin{equation}
 \bigg\|\esssup_{0\leq t \leq T_0} \int_{B_1(x_0) }|u|^2  \,dx+ \int_0^{T_0}\int_{B_1(x_0) } |\nabla u|^2\,dx\,dt \bigg\| _{l^{q/2}(x_0\in \Z^3)}  <\infty,  
\end{equation}
then there are positive universal constants $C_1$ and $\lambda_0<1$ such that
\begin{equation}
\bigg\|\esssup_{0\leq t \leq \lambda } \int_{B_1(x_0) }\frac {|u|^2}2  \,dx+ \int_0^{\lambda }\int_{B_1(x_0) } |\nabla u|^2\,dx\,dt \bigg\| _{l^{q/2}(x_0\in \Z^3)}
\leq  C_1 A_{0,q},
\end{equation}
where 
\[
A_{0,q}  =  N^0_{q} =  \bigg\|  \int_{B_1(x_0)  }  |u_0(x)|^2 \,dx   \bigg\|_{l^{q/2}(x_0\in \Z^3)} ,
\quad   \lambda=\min (T_0,\lambda_0, \frac{\lambda_0 }{A_{0,q}^{2}}). 
\]
Consequently, 
\begin{equation}\label{th2.2-2B}
\bigg\|\int_0^{\lambda }\!\!\int_{B_1(x_0) } 
|u|^{\frac {10}3} +|p-c_{x_0}(t)|^{\frac53}\,dx\,dt \bigg\|_{l^{\frac {3q}{10}}(x_0 \in \Z^3)}
\le  C A_{0,q}^{\frac 53}.
\end{equation}
\end{lemma}

It will be clear from the proof that this lemma also holds if the perturbed Navier-Stokes equations are replaced by the \emph{regularized} perturbed Navier-Stokes equations, that is, it remains valid for the solutions described in Lemma \ref{lemma.existence.loc.reg}.

\begin{proof} 

Once we establish estimates for the perturbation terms in the local energy inequality for $u$, the proof of this is identical to the proof of Lemma \ref{lem.A0qbound} with $R=1$ and $\la_R=\la$.

Let
\[
e_\la(\kappa)=\esssup_{0\leq t \leq \la} \int_{B_1(\kappa) }|u(x,t)|^2  \,dx+ \int_0^{\la}\int_{B_1(\kappa) } |\nabla u(x,t)|^2\,dx\,dt.
\]

The relevant estimates for the linear terms from the perturbed local energy inequality are bounded as
\EQ{
&\int_0^\lambda \int ( v \cdot \nb u+u\cdot \nb v ) \cdot (\phi(x-\kappa) u)\,dx\,dt 
\\&=\int_0^\lambda \int  (v \cdot \nb u) \cdot (\phi(x-\kappa) u) - u\otimes v:\nb(\phi(x-\kappa) u) \,dx\,dt 
\\&\le C \int_0^\lambda \int_{B_2(\ka)} |v|(|u|^2+|u||\nb u|) \,dx\,dt 
\\&\le C \|v\|_{L^\I L^3_\uloc}\int_0^\la   \| u\|_{L^6(B_2(\kappa))}\bke{\|  u\|_{L^2(B_2(\kappa))} +\| \nb u\|_{L^2(B_2(\kappa))} } \,dt
\\&\leq C \lambda \delta \esssup_{0<t<\lambda} \sum_{\kappa'\sim \kappa}  \int_{B_1(\kappa')} |u(x,t)|^2\,dx +C \delta \sum_{\kappa'\sim \kappa} \int_0^\lambda \int_{B_1(\kappa')} |\nb u(x,t)|^2 \,dx\,dt.
}

The pressure can be split into local and far-field parts.  The new terms of the local part are treated identically to the preceding estimates after applying the Calderon-Zygmund inequality.  The far field part of the pressure can be written as $p_{\far}=p_{\far,u}+p_{\far,v}$ where $p_{\far,u}$ is the far-field part appearing in the proof of Lemma \ref{lem.A0qbound} and  $p_{\far,v}$ is the remaining part.  The estimate for $p_{\far,u}$ is given in the proof of Lemma \ref{lem.A0qbound} and so we only need to worry about $p_{\far,v}$, which is bounded 
in $B_2(\ka)\times(0,T)$ as
\begin{align*}
|p_{\far,v}(x,t)|&\leq C \int \frac 1 {|\kappa-y|^4} |u(y,t)||v(y,t)|  (1-\chi_{4}(y-\ka))\,dy
\\&\leq C\delta  \widetilde K * e_{\la}^{1/2} {(\ka)}.
\end{align*}
This leads to 
\begin{align*}
&\int_0^{{\la  }} \int 2p_{\far,v}(x,s)  u(x,s)\cdot \nabla \phi (x-\ka) \,dx\,ds 
\\&\leq C \int_0^\la \int_{B_2(\kappa)} \de^{1/2} (\widetilde K * e_{\la}^{1/2} )\de^{1/2} |u|\,dx\,ds
\\&\leq C \delta \int_0^\la \int_{B_{2}(\ka)} (\widetilde K * e_{\la}^{1/2} ) ^{2}\,dx\,ds +   \delta \int_0^{\la }\int_{B_{2}(\ka)}|u|^2\,dx\,ds
\\&\leq C \delta  \la  ( (\widetilde K*e_{\la}^{1/2} )(\ka))^{2} +\delta \la \esssup_{0<t<\lambda} \sum_{|\kappa-\kappa'|<4}  \int_{B_1(\kappa')} |u|^2\,dx.
\end{align*} 

Combining the above estimates and using the proof of Lemma \ref{lem.A0qbound} {cf.~\eqref{eq3.11}}, we obtain
\EQ{
e_\la(\kappa)
&\leq  \int |u_0|^2\phi(x-\ka)  \,dx+C \la  \sum_{\ka'\in \Z^3; |\ka'-\ka|\leq 2} e_{\la}(\ka') 
\\&+ C      {\la^{1/4}}   \sum_{\ka'\in \Z^3; |\ka'-\ka|\leq 10} (e_{\la}(\ka'))^{3/2}
+C   {\la^{1/4}}   ( (\widetilde K*e_{\la} )(\ka ))^{3/2}  
\\&+C \delta \sum_{|\kappa-\kappa'|<10} {e_\la(\ka')}  
+ C \delta  \la^{1/4}  ( (\widetilde K*e_{\la}^{1/2} )(\ka))^{2},
}
where we are using {$\la\leq \la_0\le 1$.}
For reference, the first two lines above are identical to the estimates in the proof of Lemma \ref{lem.A0qbound}, while {terms in the last line}---those with a factor of $\de$---are new.  We therefore only discuss the remaining terms here. Applying the $l^{q/2}$ norm to $(\widetilde K * e_\la^{1/2})^{2}$ gives 
\EQ{
\|(\widetilde K*e_{\la}^{1/2} )(\ka))^{2}\|_{l^{q/2}}
&= \| (\widetilde K*e_{\la}^{1/2} )(\ka)\|_{l^{q}}^{2}
\leq  C \|\widetilde K\|_{l^1} \| e_\la \|_{l^q/2}.
}
We now choose $c_0$---which dominates $\de$---to be small enough that, when we apply the $l^{q/2}$ norm to the right hand side, the terms with a factor of $\delta$  {(finite by assumption)} will be absorbed into the $l^{q/2}$ norm of the left hand side.  The remaining terms on the right hand side are identical to those in the proof of Lemma \ref{lem.A0qbound} and the remainder of the proof is the same.
\end{proof}

\subsection{Extension for the regularized, perturbed problem}

\begin{lemma}\label{lemma.existenceRegProblem}Let $\e,\delta>0$ be given and assume $\de\le c_0$ where $c_0$ is defined in Lemma \ref{lem.A0qboundperturbed}.
Assume $u_0\in {L^2}$, is divergence free and $v:\R^3\times [0,T_0]\to \R^3$ satisfies 
$\div v=0$ and
\[
\esssup_{0<t\leq T_0}\| v(t) \|_{L^3_\uloc} <\delta.
\]
Then, there exists $T\in (0,T_0]$ and a weak solution $u_\e$ and pressure $p_\e$ to  
\Eq{
\partial_t u_\e -\Delta u_\e +\eta_\e* u_\e \cdot\nb u_\e + \eta_\e * v \cdot \nb u_\e + u_\e \cdot \nb (\eta_\e *v) +\nb p_\e = 0;\quad \nb\cdot u_\e=0,
}
on $\R^3\times [0,T]$.
Furthermore, we have $u_\e\in L^\infty (0,T ; E^2_q)$  and  satisfies 
\EQ{\label{ineq.target2}
\bigg\|\esssup_{0\leq t \leq T_0 } \int_{B_1(x_0) }\frac {|u_\e|^2}2  \,dx+ \int_0^{T_0 }\int_{B_1(x_0) } |\nabla u_\e|^2\,dx\,dt \bigg\| _{l^{q/2}(x_0\in \Z^3)}^2\leq 2C \|u_0\|_{E^2_q},
}
for a constant $C$ not depending on $\e$, $\delta$, $v$ or $u_0$. Here, $T= \min (T_0,\la_0,\la_0A_{0,q}^{-2})$ depends on $\|u_0\|_{E^2_q}$ but not on $\|u_0\|_{L^2}, u_\e, \e, \de$ or $v$.
\end{lemma}

For our application, it will be crucial that $T$ does not depend on $\|u_0\|_{L^2}$. This is consistent with the existence theory for the perturbed Navier-Stokes equations in \cite[Ch.~21]{LR} where the time-scale of existence is just the time-interval on which the perturbation factor is defined and does not depend on the {$L^2$}-size of the initial data. 

\begin{proof}
Assume that $u_\e$ and $p_\e$ are a smooth solution on $\R^3\times [0,T_0]$ with $L^2$ data. By the usual energy estimate for the perturbed equation \cite[p.~217]{LR}, it follows that $\|u_\e(t)\|_{L^2}$ is uniformly bounded on $[0,T_0]$ by some finite value $M_1$. Similarly,
 {if $u_\e\in \LE_q(0,T_0)$,} the estimates from Lemma \ref{lem.A0qboundperturbed} extend up to the time $T=\min(T_0, \la_0, \la_0 A_{0,q}^{-2})$. So, there exists $M_2$ with $\|u_\e\|_{\LE_q(0,T)}<M_2$.

Let $T_{\e,\de}$ be the timescale given in Lemma \ref{lemma.existence.loc.reg} for initial data of size $M_1$ in $L^2$ (note $T_{\e,\de}$ is independent of $M_2$). 
Then, the solution $u_\e$ from Lemma \ref{lemma.existence.loc.reg} exists on $\R^3\times [0,T_{\e,\de}]$ and belongs to $\LE_q(0,T_{\e,\de})$.
As noted earlier, Lemma \ref{lem.A0qboundperturbed} also applies to the regularized problem. We therefore conclude $\|u_\e\|_{\LE_q(0,T_{\e,\de})}<M_2$ and hence
\[
\esssup_{0<t\leq T_{\e,\delta}} \|u(t)\|_{E^2_q} 
\leq M_2.
\]
As we mentioned above, we also know 
\[
\esssup_{0<t\leq T_{\e,\delta}} \|u(t)\|_{L^2} 
\leq M_1.
\]
Hence, we can re-solve the problem at some $t_*\in [T_{\e,\delta}/2, 3T_{\e,\de}/4]$ to obtain a solution that lives on $[t_*, t_*+T_{\e,\delta}]$.  By uniqueness, this solution agrees with our original solution and we therefore obtain a solution on $[0,3/2T_{\e,\de}]$ which belongs to $  \LE_q(0,3T_{\e,\de}/2)$.
This argument can be iterated until time $T$ is reached, because Lemma \ref{lem.A0qboundperturbed} guarantees the $\LE_q$ norm of the solution is controlled by $M_2$ up to this time and we know the $L^2$ norm is controlled by $M_1$ as well.  Thus, for each $\e>0$, we obtain a solution on $[0,T]$ to the regularized problem and conclude that $u_\e \in \LE_q(0,T)$ with bound independent of $\e$.
\end{proof}

\subsection{Local existence for the perturbed problem}
 
\begin{lemma}\label{lemma.localExistence} Let $c_0$ and $\la_0$ be the constants in Lemma \ref{lem.A0qboundperturbed}. 
Assume $u_0\in E^2_q$  is divergence free, and $v:\R^3\times [0,T_0]\to \R^3$ satisfies
$\div v=0$ and  
\[
\esssup_{0<t\leq T_0}\| v(t) \|_{L^3_\uloc}<\delta\le c_0 \quad \text{and}\quad 
\esssup_{0<t\leq T_0}\| v(t) \|_{L^4_\uloc}<\I.
\]
Let $T= \min (T_0,\la_0,\la_0A_{0,q}^{-2})$.    
Then, there exists a weak solution $u$ and pressure $p$ so that  
\EQ{\label{0819a}
\partial_t u -\Delta u + u \cdot\nb u + v \cdot \nb u + u\cdot \nb v+\nb p = 0;\qquad \nb\cdot u=0,
}
in the distributional sense on $\R^3\times [0,T]$.
Furthermore, $u$ and $p$ are a local energy solution to the perturbed Navier-Stokes equations \eqref{0819a} satisfying
\Eq{\label{ineq.target3}
\norm{u}_{\LE_q(0,T)} 
\leq C \|u_0\|_{E^2_q},
}
for a constant $C$.
\end{lemma}

Above, $v=0$ is allowed in which case the solution is a local energy solution. Recall that a local energy solution to the perturbed Navier-Stokes satisfies Definition \ref{def:localLeray} with the obvious modifications when $v\neq 0$, including the \emph{perturbed} local energy inequality \eqref{ineq.plei}. Technically, the $L^4_\uloc$ quantity only needs to be finite, not small. Also, the exponent $4$ can be replaced by any $p>3$---we work with $4$ for simplicity.

\begin{proof}
Fix $u_0\in E^2_q$.  For every $\e>0$ we can find a divergence free vector field $u_0^{(\e)} \in L^2$ so that $\|   u_0-u_0^{(\e)}\|_{E^2_q}<\e$ (this can be done using the Bogovskii map which is described in \cite{Tsai-book}).  Let $u_\e$ denote the solution described in Lemma \ref{lemma.existenceRegProblem} with data $u_0^{(\e)}$. Note that our setup implies the same uniform bounds in \cite{KiSe}, as well as a uniform bound on $\partial_t u_n$  in the dual space of $L^3(0,T; W^{1,3}_0(B_M))$, and therefore we have the same convergence properties as on \cite[p.~156]{KiSe}, namely for a subsequence $u_n:=u_{\e_n}$ of $u_\e$ and $p_n:=p_{\e_n}$ of $p_\e$,
 \begin{align*}
&u_{n}\overset{\ast}\rightharpoonup u \quad \text{in }L^\infty(0,T;L^2_\loc )
\\&u_{n}  \rightharpoonup u\quad \text{in }L^2(0,T;H^1_\loc)
\\&u_n,\,\eta_{\e_n}*u_n \to u  \quad \text{in }L^3(0,T;L^3_\loc)
\\& \eta_{\e_n} * v \to v \quad \text{in }L^3(0,T;L^3_\loc)
\\&p_n^{(k)} \rightharpoonup p_k \quad \text{in }L^{3/2}(0,T;L^{3/2}(B_k))
\end{align*}
as $n\to \infty$ where $p_k(x,t) = p(x,t)-c_k(t)$ for $x\in B_k(0)$ and $t\in (0,T_0]$ for some $c_k\in L^{3/2}(0,T_0)$ and $p_n^{(k)}$ is the local pressure expansion for $p_n$ in the ball $B_k(0)$,
and  $u$ and $p$ are a local energy solution to the perturbed Navier-Stokes equations with initial data $u_0$---the fact that $u$ and $p$ are as claimed follows from the arguments in \cite{KiSe} and we omit redundant details.  We do check the \emph{perturbed} local energy inequality holds as this is used to show the solution is appropriately bounded in $\LE_q(0,T)$.  
The perturbed local energy inequality is as follows: For any nonnegative $\phi\in C_c^\I(\R^3\times {[}0,T))$, 
 \EQ{\label{ineq.plei}
&2\iint |\nabla u|^2\phi\,dx\,dt 
\\&\leq \int |u_0|^2\phi \,dx+ 
\iint |u|^2(\partial_t \phi + \Delta\phi )\,dx\,dt +\iint (|u|^2+2p)(u\cdot \nabla\phi)\,dx\,dt
\\&+ \iint  |u|^2 v\cdot\nb \phi \,dx\,dt +2 \iint (u\cdot \nb u )\cdot (v\phi)\,dx\,dt   +2\iint (u\cdot v) (u\cdot\nb\phi)\,dx\,dt,
}
where the time integrals are over the full interval $[0,T]$.  
The known compactness arguments ensure the first two lines above are inherited for $u$ from the approximation scheme. We therefore only need to address convergence of terms corresponding to the last line from \eqref{ineq.plei}. The first and last of these are lower order so we focus on the middle term. We have
\EQ{
&\bigg| \iint (u_n\cdot \nb u_n ){\cdot (\eta_{\e_n} * v) \phi  }- (u\cdot \nb u)\cdot v\phi \,dx\,dt\bigg|
\\&\leq \bigg| \iint ((u_n- u)\cdot \nb u_n)\cdot (v\phi ) \,dx\,dt\bigg| + \bigg| \iint (u\cdot \nb (u_n - u ))\cdot (v\phi )\,dx\,dt \bigg|
\\&+ \bigg| \iint (u_n\cdot \nb u_n )\cdot (\eta_{\e_n} * v-v) \phi \,dx\,dt \bigg|
\\&=: I_{1,n}+I_{2,n}+I_{3,n}.
}
We need to show that the above three quantities vanish as $n\to \I$. 
Let $B$ denote a ball containing the support of $\phi$.
We have 
\[
I_{1,n} \lesssim \|v\|_{L^\I L^4_\uloc}\|u_n-u\|_{L^2(0,T;L^2(B))}^{1/4} \| u_n, u\|_{L^2(0,T;H^1(B))}^{7/4},
\]
where we used H\"older's inequality and log-convexity of $L^p$ norms. This plainly vanishes as $n\to \I$ by strong convergence of $u_n$ to $u$ in $L^2(0,T;L^2(B))$.
The second term $I_{2,n}\to 0$ by weak convergence of $u_n$ to $u$ in $L^2 (0,T;H^1(B))$ and $u_iv_j\phi\in L^2 (B \times (0,T))$. The last term vanishes by the uniform bound of $u_n$ in $L^2 (0,T;H^1(B))$ and the strong convergence of $(\eta_{\e_n} * v )\phi $ to $v \phi$ in $L^\I (0,T;L^3(B))$. We have thus established the local energy inequality initiated at $t=0$. 
Following \cite[(3.28)-(3.29)]{KwTs}---the perturbation terms do not change the argument---we obtain the other form of the local energy inequality: for any non-negative $\psi\in C_c^\I(\R^3)$ and $t\in (0,T)$, we have 
\EQ{\label{ineq.plei2}
&\int |u(t)|^2\psi \,dx +
2\int_0^t\!\!\int |\nabla u|^2\psi\,dx\,dt 
\\&\leq \int |u_0|^2\psi \,dx+ 
\int_0^t\!\!\int |u|^2\Delta\psi \,dx\,dt +\int_0^t\!\!\int (|u|^2+2p)(u\cdot \nabla\psi)\,dx\,dt
\\&+ \int_0^t\!\!\int  |u|^2 v\cdot\nb \psi \,dx\,dt +2 \int_0^t\!\!\int (u\cdot \nb u )\cdot (v\psi)\,dx\,dt   +2\int_0^t\!\!\int (u\cdot v) (u\cdot\nb\psi)\,dx\,dt.
}

We now establish the bound for $\|u\|_{\LE_q(0,T)}$. 
Let $\phi$ be as in the proof of Lemma \ref{lem.A0qbound} and take $\psi(x) = \phi(x-k)$ in \eqref{ineq.plei2}, $k \in \Z^3$.
Assume $|k|\leq K$ for some large $K$. The right hand side of \eqref{ineq.plei2} can be obtained as the limit of terms corresponding to $u_n$; hence we can make the difference of the above terms and those for $u_n$ less than  $\frac 1 {2^K K^3}$ uniformly across all $|k|\leq K$. Taking these approximate terms,  applying standard estimates (see \eqref{th2.2-3}), taking the essential supremum in time, and summing (in the $l^{q/2}$ sense) over indexes $|k|\leq K$, we find that the right hand side is bounded independently of $n$ for $n\geq N_K$ for some sufficiently large $N_K$. This observation leads to the bound 
\EQN{
&\bigg(\sum_{|k|\leq K} \bigg(\esssup_{0<t<T} \int_{B_1(k)} |u(x,t)|^2 \,dx + \int_0^{T}\int_{B_1(k)} |\nb u|^2\,dx \,dt\bigg)^{q/2} \bigg)^{1/q}
 \leq C\|u_0\|_{E^2_q} +\frac C {2^K}. 
}
Since this is true for all $K\in \N$, it follows that $u\in \LE_q(0,T)$. The bound $\|u\|_{\LE_q(0,T)} \leq C\|u_0\|_{E^{2}_q}$ now follows the above (or from Lemma \ref{lem.A0qboundperturbed}).
\end{proof} 

\begin{remark}
\label{rmk.LR.coherence} The approximation scheme we used here is identical to that in  \cite[Ch 21  {\S3}]{LR}  {(except we do not include mollification in time on $v$)} where finite energy solutions are constructed to the perturbed Navier-Stokes equations. This means that  the solutions constructed in Lemma \ref{lemma.localExistence} can be taken to agree with the solutions constructed in  \cite[Ch 21]{LR} provided the data is in $L^2$. We may therefore use the properties of solutions in\cite{LR}  freely when working with the solutions constructed in Lemma \ref{lemma.localExistence} assuming $u_0\in L^2$.

\end{remark}

\section{Global existence of solutions in $\LE_q$}\label{sec.global} 

The typical extension argument from a local to global solution involves splitting the initial data between a small sub-critical  or critical part, which leads to small, strong solution to the Navier-Stokes equations, and an $L^2$ part, which leads to a weak solution to the perturbed Navier-Stokes equations \cite{Calderon,LR,KiSe}. The sum of these then extends the original solution on a uniform time-step. For us, we hope to  choose $u(t_0)\in E^4_q$ and decompose $u(t_0)=v_0+w_0$ where $w_0\in L^2$ and $\|v_0\|_{E^4_q}<\delta$.  By Lemmas \ref{lemma.existenceRegProblem} and \ref{lemma.E3existence}, there is a small constant $\tau_0>0$ and a local energy solution $v$ in $\LE_q(t_0,t_0+{\tau_0})   \cap L^\infty(t_0,t_0+\tau_0; L^3_\uloc)$ provided $\delta$ is sufficiently small.   There is also a solution $w$ in the local energy class to a perturbed problem so that $v+w$ is a local energy solution on $[t_0,t_0+{\tau_0}]$ with initial data $u(t_0)$.  Using Lemma \ref{lemma.E3existence}, we then glue $u$ to $w+v$ to obtain a solution to the Navier-Stokes equations on $[0,t_0+{\tau_0}]$.   

We need to additionally show that $u\in \LE_q(t_0,t_0+{\tau_0})$, and for this need to show $w\in  \LE_q(t_0,t_0+{\tau_0})$.  
The problem is that $w\in L^\infty L^2 \cap L^2H^1$ does not imply $w\in \LE_q$.  In particular, we need to have
\[
\sum_{k\in \Z^3}  \esssup_{t_0<t<t_0+{\tau_0}} \int_{B_1(k)}|u(x,t)|^2\,dx \leq \esssup_{t_0<t<t_0+{\tau_0}}  \sum_{k\in \Z^3} \int_{B_1(k)} |u(x,t)|^2\,dx = \|u\|_2^2,
\]
which is not generally true by Example \ref{example1.2}. To overcome this issue, we establish the following lemma.

\begin{lemma}\label{lemma.L2toLEq}
Let $2\le q<\infty$.
Assume $u_0\in L^2$ and is divergence free, and assume
$v:\R^3\times [0,T_0]\to \R^3$ satisfies
$\div v=0$ and
\[
\esssup_{0<t\leq T_0}\| v(t) \|_{L^3_\uloc}<\delta\le c_0 \quad \text{and}\quad 
\esssup_{0<t\leq T_0}\| v(t) \|_{L^4_\uloc}<\I,
\]
where $c_0$ from Lemma \ref{lem.A0qboundperturbed}.
For any $T \in (0,T_0]$, if $\de\le \delta_0(T)\le c_0$ is sufficiently small, then there exists a local energy solution $u$ to the perturbed Navier-Stokes equations 
\[
\partial_t u - \Delta u +u\cdot \nb u +v\cdot\nb u +u\cdot \nb v+\nb p = 0,\quad \div u=0,
\]
so that $u\in \LE_q(0,T)$.  In particular, this is true when $v\equiv 0$.
\end{lemma}

\begin{proof}
We first prove the case $v=0$ to illustrate the main idea. 
Consider initial data $u_0\in L^2$ and let $a_k = \int_{B_1(k)} |u_0|^2\,dx$ for $k \in \Z^3$. Since
\[
\textstyle
\sum_k a_k^{q/2} \le (\max a_k)^{q/2-1} \sum_k a_k \le (\sum_k a_k)^{q/2},
\]
we have $u_0\in E^2_q$ and $\norm{u_0}_{E^2_q} \le M_2:=C\norm{u_0}_{L^2}$. Now take $u$ to be the solution given by Lemma \ref
{lemma.localExistence} with $v=0$ and $u(0)=u_0$. Then $u\in \LE_q([0,T_0])$ where $T_0$ is the time-scale associated to initial data with size less than or equal to $M_2$ in $E^2_q$.  For almost every $t\in (0,T)$ we know that $u(t)\in E^3$
and $\|u(t)\|_{2}\leq \|u_0\|_2$.   That $u(t)\in E^3$ a.e.\,$t$ follows from the same argument  of \cite[Corollary 4.8]{KwTs} and the same argument
leading to \eqref{uinE4q} in the proof of Theorem \ref{thrm.existence}.
It follows that $\|u(t)\|_{E^2_q}\leq M_2 $ almost everywhere in $t$.  In particular all of these properties hold at some time $t_0\in (T_0/2,T_0)$.  We can therefore re-solve the Navier-Stokes equations starting at time $t_0$ to obtain a {local energy} solution $u_1$ that is also in $\LE_q(t_0,t_0+T_0)$. Using uniqueness of local energy solutions for data in $E^3$ \cite{BT8}, it follows that $u_1 = u$ on $[t_0,t_0+\Delta_1]$ for some small value $\Delta_1$.  This is enough to glue $u_1$ to $u$ to obtain a solution (which we still call $u$) on $[0,3T_0/2]$ that is a local energy solution and in $\LE_q(t_0,t_0+T_0) \cap \LE_q(0,T_0)$.  It is easy to see this last inclusion implies $u\in \LE_q(0,3T_0/2)$.  This argument can be iterated to construct a solution in $\LE_q(0,T)$ for any $T>0$ precisely because we have uniform in time control of $\|u\|_{E^2_q}$.

We next deal with the case when $v\neq 0$.  
We start with data $u_0\in L^2\subset E^2_q$.  Note that the  local energy solution in \cite[Theorem 21.3]{LR} where $T$ comes from the statement of Lemma \ref{lemma.L2toLEq} agrees with the $\LE_q$ solution constructed in Section \ref{sec.local} because they are both limits of the same scheme---see Remark \ref{rmk.LR.coherence}.  Denote this solution by $u$.  Then, assuming $\de\ll c_0$, $u\in \LE_q(0,T_0)$ where $T_0 = T_0(\|u_0\|_{E^2_q})$ is the time-scale of existence from Lemma \ref{lemma.localExistence}.
Using
\EQN{
 \int_{\R^3} |v||u|(|\nb u|+ |u|) &\le \sum_{k} \int_{B_1(k)} |v||u|(|\nb u|+ |u|)
\lec  \|v\|_{ L^3_\uloc} \sum_{k}  \int_{B_1(k)} (|\nb u|^2+ |u|^2), 
}we also have
\EQ{
\|u(t)\|_{2}^2 +2\|\nb u\|_{L^2(0,t;L^2)}^2 \leq \|u_0\|_2^2 + C \|v\|_{L^\infty L^3_\uloc} ( \| \nb u\|_{L^2(0,t;L^2)} ^2 + t \sup_{0<s<t} \|u(s)\|_2^2 ),
}
where $0<t<T$. Presumably $T>T_0$ or else we are done. 
By taking $\delta$ small in a fashion depending on $T$ and using $\|v\|_{L^\infty L^3_\uloc}<\delta$, we can guarantee that 
\[
\sup_{0<s<T}\|u(t)\|_{2}^2 +2\|\nb u\|_{L^2(0,T;L^2)}^2 \leq 2\|u_0\|_2^2.
\]
This implies 
\[
\sup_{0<t<T} \|u(t)\|_{E^2_q}\leq C \|u_0\|_{L^2}.
\]
This gives uniform in time control of $\|u\|_{E^2_q}$ and we proceed as when $v=0$ to complete the proof.
\end{proof}

We are now ready to prove Theorem \ref{thrm.existence}.  In it we use the following version of the $\e$-regularity of \cite{CKN} which is due to \cite{L98}; see also \cite{LS99} for details.
  
\begin{lemma}[$\e$-regularity criteria]\label{thrm.epsilonreg}  
There exists a universal small constant $\e_*>0$ such that, if the pair $(u,p)$ is a suitable weak solutions of \eqref{eq.NSE} in  $Q_r=Q_r(x_0,t_0)=B_r(x_0)\times (t_0-r^2,t_0)$, $B_r(x_0)\subset \R^3$, and    
\[
{\e^3=}\frac 1 {r^2} \int_{Q_r} (|u|^3 +|p|^{3/2})\,dx\,dt <\e_*,
\]
then $u\in L^\I(Q_{r/2})$.
Moreover,
\[
\|  \nabla^k u\|_{L^\I(Q_{r/2})} \leq C_k {\e}\, r^{-k-1},
\]
for universal constants $C_k$ where $k\in \N_0$.
\end{lemma}

\begin{proof}[Proof of Theorem \ref{thrm.existence}]  
Assume $u_0\in E^2_q$ and is incompressible.  By Lemma \ref{lemma.localExistence} with $v=0$, there exists a local energy solution $u$ to the Navier-Stokes equations on $\R^3\times (0,T_0)$ so that  $u\in \LE_q(0,T_0)$.
By Lemma  \ref{lem.A0qbound} with $R=1$, we have
\EQ{\label{E2q-exist1}
\|e_k \| _{l^{q/2}(k\in \Z^3)} \leq  C A_{0,q},\quad
\|f_k \| _{l^{q/3}(k\in \Z^3)} \leq  C A_{0,q}^{3/2},
}
where
\[
e_k = \esssup_{0\leq t \leq T_0} \int_{B_1(k) }\frac {|u|^2}2  \,dx+ \int_0^{T_0}\int_{B_1(k) } |\nabla u|^2\,dx\,dt ,
\]
\[
f_k = \int_0^{T_0}\!\!\int_{B_1(k) } 
|u|^{3} +|p-c_{k,1}(t)|^{\frac32}\,dx\,dt,
\]
and
\[
A_{0,q} =  \bigg\|  \int_{B_1(k)  }  |u_0(x)|^2 \,dx   \bigg\|_{l^{q/2}(k\in \Z^3)} = \norm{u_0}_{E^2_q}^2.
\]

Since 
\[
\lim_{R \to \I} \|f_k \| _{l^{\frac {q}{3}}(k \in \Z^3; |k|>R)}
=0,
\]
by Lemma \ref{thrm.epsilonreg},
there exists $R_0>0$ such that
\[
u \in L^\I\cap C_{\loc} (B_{R_0}^c\times [\tfrac12 T_0,T_0]), \quad 
\]
and
\[
\lim_{ R \to \I} \norm{u}_{L^\I(B_{R}^c\times [\tfrac12 T_0,T_0])} = 0.
\]
In fact, for $|k|>R_0$,
\[
\norm{u}_{L^\infty(B_1(k) \times [\tfrac12 T_0,T_0])} ^3 \le C\sum_{|k'-k|\le 2} f_{k'} .
\]
Thus
\[
\Norm{\norm{u}_{L^\infty(B_1(k) \times [\tfrac12 T_0,T_0])}}_{ l^q (k \in \Z^3; |k|>R_0)}
\le C  \|f_k \| _{l^{q/3}(k\in \Z^3; |k|>R_0-2)}^{1/3} \ll 1.
\]
By \eqref{E2q-exist1} and Sobolev imbedding,
\[
u \in L^{8/3}(0,T_0; L^4(B_{R_0})).
\]
Thus 
\Eq{\label{uinE4q}
u(t) \in E^4_q \text{ for a.e. }t \in [\tfrac12 T_0,T_0].
}
 	
Choose $t_0 >T_0/2$ such that $u(t_0) \in E^4_q$. 
We next construct a local energy solution in $\LE_q(t_0,t_0+\tau_0)$ with initial data $u(t_0)\in E^4_q$, where $\tau_0$ is the fixed time-scale in Lemma \ref{lemma.E3existence}. For any $\de>0$, we can split $u(t_0)$ into a sum of divergence free vector fields $v_0$ and $w_0$,
\[
u(t_0) = v_0 + w_0,\quad \div v_0 = \div w_0=0,
\]
where 
\[
\|v_0\|_{E^4_q}<\de,  \quad w_0\in L^2(\R^3).
\] 
Roughly, $v_0$ should be the tail of $u_0$ and $w_0$ the core.  The splitting can be carried out using the Bogovskii map (see \cite{Tsai-book} for details on the map and \cite{BT5} for a similar application).

By Lemma \ref{lemma.E3existence} and taking $\de$ sufficiently small, there exists a local energy solution $v$ with initial data $v_0$ and pressure  $p$ on  $\R^3\times (t_0,t_0+\tau_0)$ so that $v$ and $p$ are smooth in space and time and
\[
\sup_{t_0\leq t \leq t_0+\tau_0} \|v(t)\|_{L^4_\uloc}
<C \de.
\] 
By the uniqueness assertion of Lemma \ref{lemma.E3existence}, $v$ agrees with the solution given by Lemma \ref{lemma.localExistence}  (we may assume $\tau_0 \le \la_0$; the solution coming from Lemma \ref{lemma.localExistence} is here taken with zero perturbation), and hence $v \in \LE_q(t_0,t_0+\tau_0)$, too.

By Lemma \ref{lemma.L2toLEq} where the perturbation factor is $v$ and again taking $\delta$ sufficiently small to ensure the timescale generated by  Lemma \ref{lemma.L2toLEq} is larger than $\tau_0$, there exists a local energy solution $w$ to the perturbed Navier-Stokes equations in $ \LE_q(t_0,t_0+\tau_0)$ with initial data $w_0$ and an associated pressure $\pi$.  Letting $u_1= v+w$ gives a local energy solution on $\R^3\times (t_0,t_0+\tau_0)$.  Note that to obtain the local energy inequality for $u_1$ we need to use  the local energy inequality for $v+w^{(n)}$ where $w^{(n)}$ approximates $w$ as in the proof of Lemma \ref{lemma.localExistence} (see also \cite{BT6,AB}).

Since $u$ and $u_1$ agree at $t_0$ and $u(t_0)\in E^3$ since $E^4\subset E^3$, \cite[Theorem 1.6.b]{BT8} implies, for some $\ga>0$, that  $u(x,t)=u_1(x,t)$ on $\R^3\times (t_0,t_0+\ga)$.  Hence, we may glue $u_1$ to $u$ to obtain a local energy solution $u\in \LE_q (0,t_0+\tau_0)$. Repeating this procedure $n$ times leads to a  solution $u\in \LE_q(0,t_0+ n\tau_0)$. Letting $n\to \I$  yields the solution described in Theorem \ref{thrm.existence}.
\end{proof}

\section{Appendix: NSE with $E^p$ data, $p>3$}\label{sec.appendixE3}

In this section we state and prove a lemma on the existence of smooth solutions for data in $E^4$.  Note that $E^4$ can be replaced with $E^p$ for any $p>3$, and we choose $p=4$ for simplicity. The ideas used are standard and the conclusion is no doubt known. However, we cannot find the statement and proof of what we need  exactly in the literature.  Most elements of a proof can be pieced together from \cite{LR}. Additionally, a similar lemma without proof is given in \cite{KiSe}.  Our proof relies on the local energy methods in \cite{BT8} and the mild solution theory in \cite{MaTe}.

\begin{lemma}\label{lemma.E3existence}
Suppose $a\in E^4$   and is divergence free.   
Assume also that \[\delta:=  \| a\|_{L^4_\uloc} < \e_*,\] for a universal constant $\e_*$.  There exists a second universal constant $\tau_0>0$ and  $v$ and $p$ comprising a local energy solution to \eqref{eq.NSE} in $\R^3\times (0,\tau_0)$ with initial data $a$ so that $v$ and $p$ are smooth in space and time, $v\in C([0,\tau_0];E^4) $  and   
\[
	\sup_{0\leq t \leq \tau_0} \|v(t)\|_{L^4_\uloc} <C \de.
\] 
Furthermore, if $u\in \mathcal N(a)$ then $u=v$ on $\R^3\times [0,\tau_0]$.
\end{lemma}

That $v\in C([0,\tau_0];E^p)$ means $v\in C([0,\tau_0];L^p_\uloc)$ and $v(t) \in E^p$ for all $t$.
To prove Lemma \ref{lemma.E3existence} we will need the following result from \cite{BT8}. Compared to \cite{MaTe}, this is a weak-strong  uniqueness criteria as opposed to a strong-strong uniqueness criteria.
\begin{theorem}[\cite{BT8} Theorem 1.7]\label{thrm.uniquenessBT8}
Assume $a\in E^2$ and is divergence free. Let $u,\,v\in \mathcal N(a)$.  There exist universal constants $0<\e_3, \tau_0 \leq 1$ so that, if 
\[
\sup_{0<r\leq R} N^0_r \leq \e_3
\]
for some $R>0$,
then $u=v$ as distributions on $\R^3\times (0,T)$, $T= \tau_0 R^2$.
\end{theorem}
See also \cite{Jia-uniqueness}.
Above, the constant of proportionality $\tau_0$ is the same constant appearing in Lemma \ref{lemma.E3existence}.

\begin{proof}[Proof of Lemma \ref{lemma.E3existence}]
 Since $a\in E^4$, $a\in E^2$ and there exists a global in time local energy solution $u$ evolving from $a$ (see \cite{LR, KiSe,KwTs}). Now, assume $\|a\|_{L^3_\uloc}<\e_3$ (this is fine because $\|a\|_{L^3_\uloc} \lesssim \|a\|_{L^4_\uloc}$). Then, for all $r\leq 1$ and all $x_0\in \R^3$,
\[
\frac 1 r \int_{B_r(x_0)} |a|^2\,dx \leq C\bigg(  \int_{B_r(x_0)}|a|^3\,dx \bigg)^{2/3} \leq {C}\e_3.
\]
Hence, by further restricting our time-scale and using Theorem \ref{thrm.uniquenessBT8} we have uniqueness in $\mathcal N(u_0)$ up to time $\tau_0$.

By \cite{MaTe}, since $E^4\subset \mathcal L^4_\uloc$ (the closure of $BUC(\R^3)$ in $L^4_\uloc$-norm), there exists a time-scale $T$ and a unique mild solution $v\in C ([0,T);L^4_\uloc )$ that is smooth on $\R^3\times (0,T)$.
~By taking $\e_*$ sufficiently small, we can ensure the existence time in \cite[Theorem 1.1]{MaTe} is greater than $\tau_0$.

We need to show that $v$ is a local energy solution. By embeddings, the same convergence properties at $t=0$ hold with $L^4$ and $L^4_\uloc$ replaced by $L^2$ and $L^2_\uloc$. This implies that, if $w\in L^2(\R^3) $
is compactly supported, then
\[
\lim_{t\to 0} \int (v(x,t)-a(x))w(x)\,dx = 0.
\]
Continuity of the map
\[
t\mapsto \int v(x,t)w(x)\,dx
\]
at positive times follows from smoothness of $v$ in the space and time variables.  
 
We use \cite{BT7} to recover a pressure $p$ satisfying the local pressure expansion
so that $v$ and $p$ solve
\[
\partial_t v -\Delta v +v\cdot\nb v+\nb p = 0,
\]
as distributions. 
Furthermore, the local expansion for $p$ ensures that $p\in L^{3/2}_\loc(\R^3\times (0,T))$.
The local energy inequality follows from the fact that $v$ and $p$ are smooth in the space and time variables.
Finally, using the local energy inequality and the fact that $v\in L^\I L^3_\uloc$, we obtain item 2 from the definition of local energy solutions.  This proves that $v\in \mathcal N(u_0)$.  Uniqueness then implies $u=v$ on $\R^3\times (0,\tau_0)$ and, therefore,
\[
\| u (t)\|_{L^3_\uloc}\leq C\| u (t)\|_{L^4_\uloc}< C \delta,
\]
for all $t\in (0,\tau_0)$.
 
The solution from \cite{MaTe} belongs to $C([0,\tau_0];L^4_\uloc)$ (this is because $E^4\subset \mathcal L^4_\uloc$---see \cite{MaTe} for the definition of $\mathcal L^4_\uloc$).  By far field regularity of local energy solutions with data in $E^2$, we have $u(t)\in E^4$ for all $t>0$---see e.g.~the proof of Theorem \ref{thrm.existence}. This implies  $u\in C([0,\tau_0];E^4)$.
\end{proof}

Lemma  \ref{lemma.E3existence} is not optimal in the sense that some assumptions can be weakened to yield similar results.  It is, however, necessary and sufficient for our purposes and easy to prove.  In contrast, the following statement is given in \cite{KiSe} without proof.

\begin{lemma}\label{lemma.E3existence2}
Suppose $a\in E^3$ and is divergence free.  Then there exists ${\tau_1}>0$ and  functions $v$ and $p$ comprising a weak solution to \eqref{eq.NSE} in $\R^3\times (0,{\tau_1})$ so that 
\[
	v\in C([0,{\tau_1}];E^3)\text{ and }t^{1/2}v\in L^\I(\R^3\times (0,{\tau_1}) ).
\] 
\end{lemma}

Unlike $\tau_0$ in Lemma \ref{lemma.E3existence}, $\tau_1$ depends on $a$ and is not universal.
While not proven in \cite{KiSe}, this can be proven using results in \cite{LR,KochTataru,MaTe} as well as our proof of Lemma \ref{lemma.E3existence} above. In particular, 
note that $L^3_\uloc$ embeds continuously in $bmo^{-1}$ (see \cite[Proposition 2.1]{MaTe}; continuity is clear upon inspecting the proof).  Also, $E^3\subset vmo^{-1}$ because 
\EQ{
\frac 1 {|B(x,R)|} \int_{B(x,R)}\int_0^{R^2} |e^{t\Delta }f(y)|^2\,dt\,dy
&\leq \frac 1 {R^2} \int_0^{R^2} \bigg( \int_{B(x,R)} |e^{t\Delta}f|^3\,dy\bigg)^{2/3}\,dt,
}  
which vanishes as $R\to 0$ whenever $f\in E^3$.
Because $a\in vmo^{-1}$, there exists a mild solution $v$ on $[0,T]$ for some $T>0$ \cite{KochTataru}. This is a local energy solution and will coincide with any other local energy solution by the uniqueness result in \cite{BT8}---for this we argue as in our proof of Lemma \ref{lemma.E3existence}. By the first remark on \cite[p.~182]{LR}  and   \cite[Theorem 18.3]{LR}, we have that $v\in C( (0,T); L^3_\uloc))$.  Furthermore, $\lim_{t\to 0} \| v -e^{t\Delta}a\|_{L^3_\uloc}=0$, which follows from the fact that $a\in vmo^{-1}$ (which Lemari\'e-Rieusset refers to as $cmo^{-1}$; in \cite[Definition 17.1 (B)]{LR} Lemari\'e-Rieusset specifies what is meant by ``smooth element of'' a shift invariant Banach space) and the second remark on \cite[p.~182]{LR}.   Because $a\in E^3$, we have $\lim_{t\to 0}\|a-e^{t\Delta}a\|_{L^3_\uloc}=0$ (see \cite[Proposition 2.2]{MaTe}).  Therefore, $v\in C(  [0,T];L^3_\uloc)$.

\section{Appendix: Relations between function spaces}\label{sec.appendix}

In this appendix we include several observations that relate $L^{3,\I}$ to the $E^2_q$ space.  These are motivated by the possibility that the $E^2_q$ spaces will prove useful when analyzing solutions with data in $L^{3,\I}$ as they more accurately capture the decay at spatial infinity of the solution than existing local energy estimates.

 \begin{lemma}\label{lem.alphasum}
If $u_0\in L^{3,\I}$, then $u_0\in E^2_q $ for every $q>3$. 
 \end{lemma}

This fails for $q\le 3$ since 
$\frac 1{|x|} \in L^{3,\I}$ and $\frac 1{|x|} \not \in E^2_q$ for all $q \le 3$.
 \begin{proof}
Let  $\al_k=\int_{B_1(k)} |u_0|^2\,dx$ for all $k\in \Z^3$.  
Let \[
E(\si)=\cup_{\{k\in \Z^3:\, \al_{k}>\si\}}  B_1{(k)}.
\]
Then, $|E(\si)|\simeq\# \{ k\in \Z^3:\al_k>\si \}$.
Since $u_0\in L^{3,\I}$,
\[
\frac 1 {|E(\si)|^{1/3}}\int_{E(\si)} |u_0|^2\,dx \leq C \|u_0 \|_{L^{3,\I}}^{{2}},
\]
for all $\si>0$. As the left side is greater than $C\si |E(\si)|^{2/3}$, we get
\[
|E(\si)| \le C\si^{-2/3}  \|u_0 \|_{L^{3,\I}}^3, \quad \forall \si>0.
\]
Then, letting $S_j = \{   k\in \Z^3: 2^{j}\leq \al_k <2^{j+1}  \}$, we have
\[
\sum_{k\in \Z^3} \al_k^{q/2}= \sum_{j\in \Z} \sum_{k\in S_j} \al_k^{q/2}.
\]
Since $\al_k\in L^\I(\Z^3)$ and $|S_j|\lesssim |E(2^j)|<\infty$,
$\sum_{j\geq 0 }   \sum_{k\in S_j} \al_k^{q/2}$ is finite.
On the other hand, since $|S_j|\lesssim |E(2^j)|$,
\[\sum_{j< 0 }   \sum_{k\in S_j} \al_k^{q/2}  \leq \sum_{j<0}  |S_j|  2^{(j+1)q/2}  \leq C   \|u_0 \|_{L^{3,\I}}^{{3}}\sum_{j<0} \frac {2^{(j+1)q/2}} { 2^{3j/2}},\]
which converges whenever $q>3$.
Therefore, $u_0\in E^{2}_q$ for all $q>3$.
 \end{proof}

The next lemma shows how to connect critical, summed quantities at all scales to the $L^{3,\I}$ norm of the initial data. When $u\in \mathcal N(u_0)\cap \LE_q$ and $3<q<6$, this lemma and Theorem \ref{thrm.boundE2q} give estimates on $u$ and its decay at spatial infinity directly in terms of the $L^{3,\I}$ norm of the initial data.   Recall $A_{0,q}(R)= \bke{\int_{B_R(Rk)}|u_0|^2}_{l^{q/2}(k\in \Z^3)}$ is defined in Lemma \ref{lem.A0qbound}.
 \begin{lemma}\label{lem.A0q}
If $u_0\in L^{3,\I}(\R^3)$, then, for any $q>3$, $A_{0,q}(R)\leq CR\|u_0\|_{L^{3,\I}}^2 $ for all $R>0$. 
 \end{lemma}
 \begin{proof}

For $k\in \Z^3$, let $\be_k = \|u_0\|_{L^{3,\I} (B_{R}(kR))}^3$.  Then,
 \begin{align*}
\si^3 \#\{k\in \Z^3 : \be_k >\si^3 \}&\leq  \sum_{k:\be_k>\si}\be_k 
\\&\leq \sum_{k:\be_k>\si}\sup_{\tau>0}\tau^3 |\{ x\in B_{R}(Rk) :  |u_0(x)|>\tau  \}| 
\\&\lec \|u_0\|_{L^{3,\I}}^3.
 \end{align*}
So,
\[
\si \#\{k\in \Z^3: \be_k^{1/3}>\si  \}^{1/3}\lec \|u_0\|_{L^{3,\I}},
\]
i.e., $
\|\be_k^{1/3} \|_{l^{3,\I}}\lec \|u_0\|_{L^{3,\I}}.
$
It follows that  $\|\be_k^{1/3} \|_{l^{q}} \lec  \|u_0\|_{L^{3,\I}}$ for all $q>3$. 
Also note that
\[
\int_{B_R(Rk)}|u_0|^2\,dx\leq CR   \|u_0 \|_{L^{3,\I} (B_R(Rk)) }^2=CR\be_k^{2/3}.
\]
Raising both sides to the $q/2$ power and summing over $k\in \Z^3$ proves that  $A_{0,q}(R)\leq CR \|u_0\|_{L^{3,\I}}^2$ for all $R>0$.  
 \end{proof}

\section*{Acknowledgments}
The research of Tsai was partially supported by the NSERC grant RGPIN-2018-04137. Z.~Bradshaw was supported in part by the Simons Foundation (635438).
 
\addcontentsline{toc}{section}{\protect\numberline{}{References}}

 Zachary Bradshaw, Department of Mathematics, University of Arkansas, Fayetteville, AR 72701, USA;
 e-mail: zb002@uark.edu
 \medskip
 
 Tai-Peng Tsai, Department of Mathematics, University of British
 Columbia, Vancouver, BC V6T 1Z2, Canada;
 e-mail: ttsai@math.ubc.ca

\end{document}